\documentclass[11pt,leqno]{article}
\usepackage{amsmath}
\usepackage{rotating}
\usepackage{theorem}
\usepackage{amssymb}
\usepackage{epic}
\usepackage{curves}

\usepackage{amsmath, amsfonts, mathrsfs,latexsym,color,array}

\usepackage{cite}
\usepackage{url}
\usepackage{a4}

\numberwithin{equation}{section}

\def\del{\partial}

\newcommand{\nint}{{{\int \kern-1.13em {\begin{turn}{-20}$\bigm /$%
\end{turn}}\!\!\!}}}
\newcommand{\notint}
{{{\,\int \kern-1.01em \raise1pt\hbox{{\begin{turn}{-30}$/$%
\end{turn}\!\!}}}}}

\newcommand{\NotInt}{{\setlength{\unitlength}{.3mm}
\begin{picture}(15,15)(0,0)
\put(-2,0){\line(2,1){12}}
\put(0,0){\makebox{$\int$}}
\end{picture}\kern-2mm}}

\newcommand{\NOTINT}{{\setlength{\unitlength}{.3mm}
\begin{picture}(15,15)(0,0)
\put(0,0){\line(2,1){12}}
\put(0,0){\makebox{$\displaystyle{\int}$}}
\end{picture}\kern-2mm}}

\newcommand\qed{\hfill\vrule height8pt width6pt depth0pt}

\def\res{
\hbox{ {\vrule height .22cm}{\leaders\hrule\hskip.2cm} }
}

\renewcommand{\d}{\delta}

\newcommand{\diam}{\,\mathrm{diam}\,}
\newcommand{\dist}{\,\mathrm{dist}\,}

\newcommand{\e}{\epsilon}
\newcommand{\G}{\Gamma}
\newcommand{\g}{\gamma}
\renewcommand{\H}{\mathcal H}

\newcommand{\lip}{\mathrm{Lip}}

\newcommand{\NN}{\mathbb{N}}
\renewcommand{\O}{\Omega}

\renewcommand{\o}{\omega}

\newcommand{\po}{\partial\O}
\newcommand{\RR}{\mathbb R}

\newcommand{\p}{\partial}
\newcommand{\spt}{\mathrm{spt}\;}

\newcommand{\lan}{\langle}
\newcommand{\ran}{\rangle}

\newcommand{\beq}{\begin{equation}}
\newcommand{\bea}{\begin{eqnarray}}

\newcommand{\eeq}{\end{equation}}
\newcommand{\eea}{\end{eqnarray}}

\theoremstyle{plain}
\newtheorem{thm}{Theorem}[section]
\newtheorem{lem}{Lemma}[section]

\newtheorem{prop}{Proposition}[section]
\newtheorem{cor}{Corollary}[section]

\begingroup \makeatletter
\gdef\th@special{\normalfont\itshape
  \def\@begintheorem##1##2{%
        \item[\hskip\labelsep \theorem@headerfont ##1\ ##2.]}%
\def\@opargbegintheorem##1##2##3{%
   \item[\hskip\labelsep \theorem@headerfont ##1\ ##2\ (##3)]}}
\makeatother
\endgroup
\theoremstyle{special}

{\theorembodyfont{\rmfamily}

\newtheorem{remark}{Remark}[section]

\newtheorem{defn}{Definition}[section]

\newcommand{\cM}{\mathcal M}
\newcommand{\cF}{\mathcal F}

\newenvironment{proof}{\medskip\noindent{\bf Proof.}}{\medskip}

\newcommand{\Tan}{\mathrm{Tan}\,}

\begin{document}

\title{Boundary structure and size in terms of interior and exterior harmonic measures in higher 
dimensions}
\author{C. Kenig\footnote{The first author was partially supported by NSF grant DMS-0456583.}\and  D. Preiss  
\and T. Toro
\footnote{The third author was partially supported by NSF grant DMS-0600915 }}
\date{}
\maketitle

\setcounter{section}{0}

\begin{abstract}{In this work we introduce the use of powerful tools from geometric measure
theory (GMT) to
study problems related to the size and structure of sets of mutual absolute continuity for the
harmonic measure $\o^+$ of a domain $\O=\O^+\subset\RR^n$ and the harmonic
measure $\o^-$ of $\O^-$, $\O^-=\mbox{int}(\O^c)$, in dimension $n\ge 3$.}
\end{abstract}

\section{Introduction}

In this work we introduce the use of powerful tools from geometric measure
theory (GMT) (mainly coming from \cite{P}, in which Preiss proved that if
the $m$-density of a Radon measure $\mu$ in $\RR^n$, exists and is
positive and finite, for $\mu$-almost every point of $\RR^n$, then $\mu$ is
$m$-rectifiable, see \cite{M} for all the relevant definitions) to
study problems related to the size and structure of sets of mutual absolute continuity for the
harmonic measure $\o^+$ of a domain $\O=\O^+\subset\RR^n$ and the harmonic
measure $\o^-$ of $\O^-$, $\O^-=\mbox{int}(\O^c)$, in dimension $n\ge 3$.
These GMT tools are combined with the blow-up analysis developed by
Kenig-Toro \cite{KT2}, the properties of harmonic functions on
non-tangentially accessible (NTA) domains \cite{JK} and the monotonicity
formula of Alt-Caffarelli-Friedman \cite{ACF} to obtain analogs for $n\ge
3$ of some well-known results when $n=2$.

Let us first briefly describe some of the 2-dimensional results. Thus, let
$\O\subset\RR^2$ be a simply-connected domain, bounded by a Jordan curve
and let $\o$ be the harmonic measure associated to $\O$ (see \cite{GM}).
Then we can write  $\po$ as a disjoint union, with the following
properties:
\begin{equation}\label{intro-decomp}
 \po=G\cup S\cup N
\end{equation}
\begin{enumerate}
\item[i)] $\o(N)=0$.
\item[ii)] In $G$, $\o\ll \H^1\ll \o$, where $\H^s$ denotes $s$ dimensional
Hausdorff measure.
\item[iii)] Every point of $G$ is the vertex of a cone in $\O$. Moreover if $C$
denotes the set of ``cone points'' of $\po$, then $\H^1(C\backslash G)=0$
and $\o(C\backslash G)=0$.
\item[iv)] $\H^1(S)=0$.
\item[v)] $S$ consists ($\o$ a.e.) of ``twist points'' (a geometrical
characterization of $S$). See \cite{GM} for the definition of twist point.
\item[]vi) For $\o$ a.e. $Q\in G$ we have that
\[
\lim_{r\to 0}\frac{\o(B(Q,r)\cap\po)}{r}=L\ \hbox{ exits and }\ \ 0<L<\infty. 
\]
\item[]vii) At $\o$ a.e. point $Q\in S$ we have
\begin{eqnarray*}
\limsup_{r\to 0} \frac{\o(B(Q,r)\cap\po)}{r} & = & +\infty \\
\liminf_{r\to 0} \frac{\o(B(Q,r)\cap\po)}{r} & = & 0
\end{eqnarray*}
\end{enumerate}

These results are  a combination of work of Makarov, McMillan, Pommerenke
and Choi. See \cite{GM} for the precise references. 

Recall that the Hausdorff dimension of $\o$ (denote by $\H-\dim\o$) is
defined by
\begin{eqnarray}\label{eqn-TS.A.33}
\H-\dim\o &=& \inf\,\{k:\mbox{ there exists } E\subset\po 
\mbox{ with }\H^k(E)=0\mbox{ and }\\
&\ &  \o(E\cap K)=\o(\po\cap K) 
\mbox{ for all compact sets } K\subset\RR^n\}\nonumber
\end{eqnarray}

Important work of Makarov \cite{Mak} shows that for simply connected domains in $\RR^2$
$\H-\dim\o=1$, establishing Oksendal's conjecture in dimension 2. Carleson \cite{C}, and Jones and 
Wolff \cite{JW} proved in general for domains in $\RR^2$ with a well defined harmonic measure $\o$,
$\H-\dim\o\le1$. T. Wolff \cite{W} showed, by a deep
example, that, for $n\ge 3$, Oksendal's conjecture ($\H-\dim\o=n-1$) fails. He
constructed what we will call ``Wolff snowflakes'', domains in $\RR^3$,
for which $\H-\dim\o>2$ and others for which $\H-\dim\o<2$. In
Wolff's construction, the domains have a certain weak regularity property,
they are non-tangentially accessible domains (NTA), in the sense of
\cite{JK}, in fact, they are 2-sided NTA domains (i.e. $\O$ and $\mbox{int}(\O^c)$
are both NTA) and this plays an important role in his estimates. Here,
whenever we refer to a ``Wolff snowflake,'' we will mean a  2-sided NTA
domain in $\RR^n$, for which $\H-\dim\o\ne n-1$. In \cite{LVV}, Lewis,
Verchota and Vogel reexamined Wolff's construction and were able to
produce ``Wolff snowflakes'' in $\RR^n$, $n\ge 3$, for which both
$\H-\dim\o^\pm>n-1$, and others for which $\H-\dim\o^\pm<n-1$.
They also observed, as a consequence of the monotonicity formula in \cite{ACF},
that if $\o^+\ll\o^-\ll\o^+$ then $\H-\dim\o^\pm\ge n-1$.

Returning to the case of $n=2$, when $\O$ is again simply connected,
bounded by a Jordan curve, $\o^+=\o$ and $\o^-$ equals the harmonic measure for
$\mbox{int}(\O^c)$, Bishop, Carleson, Garnett and Jones \cite{BCGJ} showed that,
if $E\subset\po$, $\o^+(E)>0$, $\o^-(E)>0$, then $\o^+\perp\o^-$ on $E$ if
and only if $\H^1(Tn(\po)\cap E)=0$, where $Q\in Tn(\po)\subset\po$ 
if $\po$ has a unique tangent line at $Q$.
Recall that $\po$ admits a decomposition relative to $\o^\pm$,
$\po=G^\pm\cup S^\pm\cup N^\pm$  (see (\ref{intro-decomp})).
Let $E\subset \po$ be such that $\o^+\ll\o^-\ll\o^+$ on $E$ and $\o^\pm(E)>0$, then, because of 
\cite{BCGJ}, modulo sets of $\o^\pm$ measure $0$, $E\subset Tn(\po)$. 
Using Beurling's inequality, i.e. the fact that for 
$Q\in\po$ and $r>0$, $\o^+(B(Q,r))\o^-(B(Q,r))\le C r^2$, and the characterization above for $G^\pm$ and $S^\pm$ (see ii), vi) and vii)) we conclude that $\o^+\ll\H^1\ll\o^-\ll\o^+$
on $E$. Thus, sets of mutual absolute continuity of $\o^-$, $\o^+$ are ``regular''
and hence obviously of dimension 1.

In \cite{B}, motivated by this last result,  Bishop asked 
whether in the case of
$\RR^n$, $n\ge 3$, if $\o^-, \o^+$ are mutually absolutely continuous on a
set $E\subset\po$, $\o^\pm(E)>0$, then $\o^\pm$ are
mutually absolutely continuous with $\H^{n-1}$ on $E$ (modulo a set of $\o^\pm$ measure zero)  
and hence $\dim_{\H}(E)=n-1$. On the other hand, Lewis, Verchota and Vogel \cite{LVV}
conjectured that there are ``Wolff snowflakes'' in $\RR^n$, $n\ge 3$ with 
$\H-\dim\o^\pm>n-1$, for
which $\o^+$, $\o^-$ are not mutually singular. In this paper we
study these issues, for domains which verify the weak regularity hypothesis
of being 2-sided locally NTA, (a condition which, of course, Wolff
snowflakes verify). This condition ensures that we have scale invariant estimates for
harmonic measures. In the $n=2$ case, the condition is equivalent to locally
being a quasi-circle, but it is weaker than that when $n>2$. We expect that
versions of our results will still be valid under even weaker regularity
assumptions. We would like to stress though that no flatness assumption is
made in this work, and that has been one of the main points that we wanted to
address here, as well as the new introduction of the techniques from GMT
(\cite{P}), combined with the blow-up analysis in \cite{KT2}.

Our main result is that, for $n\ge 3$, $\po=\Gamma^\ast\cup S\cup N$, where
$\o^+\perp\o^-$ on $S$, $\o^\pm(N)=0$, and on $\Gamma^\ast$, $\o^+\ll
\o^-\ll\o^+$, $\dim_{\H}\Gamma^\ast\le n-1$ and if $\o^\pm(\Gamma^\ast)>0$,
$\dim_{\H} \Gamma^\ast=n-1$, where $\dim_{\H}$ denotes the Hausdorff dimension of a set.
As a consequence there can be no ``Wolff
snowflake'' for which $\o^+$, $\o^-$ are mutually absolutely continuous. We
also show that $\G^\ast=\G^\ast_g\cup\G^\ast_b\cup Z$, $\o^\pm(Z)=0$, where in
$\G^\ast_g$, $\H^{n-1}$ is $\sigma$-finite, $\o^-\ll \H^{n-1}\ll\o^+\ll\o^-$,
while on $\G^\ast_b$, for a Borel set $E$ we have that if  $\o^\pm(\G^\ast_b\cap
E)>0$ then $\H^{n-1}(\G^\ast_b\cap E)=+\infty$.
If in addition we assume that $\H^{n-1}\res\po$ is a Radon measure, we show
that $\G^\ast$ is $(n-1)$ rectifiable. In this case, we must have
$\o^\pm(\G^\ast_b)=0$, and hence $\po=\G^\ast_g\cup S\cup\widetilde N$,
$\o^\pm(\widetilde N)=0$ and $\G^\ast_g$ is rectifiable.

Our approach is the following. Using the blow-up analysis developed in
\cite{KT2}, at $\o^\pm$ a.e point on the set where $\o^+$ and $\o^-$ are mutually
absolutely continuous, the
tangent measures to $\o^\pm$ (in the sense of \cite{P}, \cite{M}) are
harmonic measures associated to the zero set of a harmonic polynomial (see Theorem \ref{thm-structure.2}).
Using the fact that for almost every point a tangent measure to a tangent measure is a tangent
measure, (see \cite{M}) and the fact that the zero
set of a harmonic polynomial is smooth except for a set of Hausdorff dimension $n-2$ (see \cite{HS}), one shows
that at $\o^\pm$ a.e. point on this set, $(n-1)$ flat
measures always arise as tangent measures to $\o^\pm$. They correspond to linear
harmonic polynomials. We then show, and this is the crucial step, that if
one tangent measure is flat, on the set of mutual absolute
continuity, then all tangent measures are flat (see Theorem \ref{lem-TS.A.3}). To accomplish this
we use a connectivity argument from \cite{P}. The key point is that if a tangent measure
is not flat, being the harmonic measure associated to the zero set
of a harmonic polynomial of degree higher than 1, its tangent measure at
infinity is far from flat ( see Lemma \ref{lem-TS.2}), and a connectivity argument
in $d$-cones of measures, in the metric introduced by Preiss in \cite{P},
gives a contradiction. 
Modulo a set of $\o^\pm$ measure 0, let $\G^\ast$ be the points in the set 
of mutual absolute continuity for which
one (and hence all) tangent measures are $(n-1)$ flat. An easy
argument (see Lemma \ref{lem-size.2} and the proof of Theorem \ref{lem-TS.A.3}) shows 
that $\dim_{\H}\G^\ast\le n-1$. To
conclude that if $\o^\pm(\G^\ast)>0$, $\dim_{\H}\G^\ast=n-1$, one uses the
Alt-Caffarelli-Friedman monotonicity formula, of \cite{ACF} as in \cite{LVV} . If $\H^{n-1}\res\po$
is a Radon measure, one can show that its density on $\G^\ast$ is $1$,
$\H^{n-1}\res\po$ a.e., which shows that $\G^\ast$ is $(n-1)$
rectifiable (see \cite{M}).

We believe that the techniques we use in showing that if one tangent measure
to $\o^\pm$ is $(n-1)$ flat, all of them are, at points of mutual absolute
continuity, should also prove useful in other situations.

{\bf Acknowledgment}: We are grateful to J.\ Garnett for his detailed
explanation of the 2 dimensional results mentioned here.

\section{Some results in geometric measure theory } 

We start this section with some basic definitions in GMT. Then we recall two families of
``distances'' between Radon measures in Euclidean space which are
compatible with weak convergence. They were initially introduced in \cite{P}.
We finish the section with a general
theorem whose consequences yield several results concerning the
structure of the boundary of a domain based on the relative behavior of 
interior harmonic measure with respect to exterior harmonic measure.

Recall that if $\Phi$ is a Radon measure in $\RR^n$
\begin{equation}\label{eqn-GMTP.1}
\spt\Phi=\{x\in\RR^n:\Phi(B(x,r))>0\ \ \ \forall\,r>0\}.
\end{equation}

\begin{defn}\label{defn-GMTP.1}
Let $\Phi$ and $\Psi$ be Radon measures in $\RR^n$. Let $K$ be a compact
set in $\RR^n$ define
\begin{enumerate}
\item[i)]$F_K(\Phi)=\displaystyle\int\dist(z,K^c)\,d\Phi(z)$.
\item[ii)]If $F_K(\Phi)+F_K(\Psi)<\infty$, let
\end{enumerate}
\begin{equation}\label{eqn-GMTP.2}
F_K(\Phi,\Psi)=\sup\{\left|\int fd\Phi-\int fd\Psi\right|:\spt f\subset K, f\ge 0,
\lip f\le 1\}.
\end{equation}
We denote by $F_r(\Phi)=F_{B(0,r)}(\Phi)$. Note that $F_K(\Phi)=F_K(\Phi,0)$.
\end{defn}

\begin{remark}\label{rem-GMTP.1}
Let $\Phi$ be a Radon measure in $\RR^n$. For $x\in\RR^n$ and $r>0$ define
$T_{x,r}:\RR^n\to\RR^n$ by the formula $T_{x,r}(z)=(z-x)/r$. Note that: 
\begin{enumerate}
\item[i)]$T_{x,r}[\Phi](B(0,s)):=\Phi(T^{-1}_{x,r}(B(0,s))=\Phi(B(x,sr))$ 
for every $s>0$.
\item[ii)]$\displaystyle\int f(z)dT_{x,r}[\Phi](z)=\displaystyle\int f\left(\frac{z-x}{r}\right)
d\Phi(z)$ whenever at least one of these integrals is defined
\item[iii)]$F_{B(x,r)}(\Phi)=r F_1(T_{x,r}[\Phi])$
\item[iv)]$F_{B(x,r)}(\Phi,\Psi)=r F_1(T_{x,r}[\Phi], T_{x,r}[\Psi])$
\end{enumerate}
\end{remark}

\begin{defn}\label{defn-GMTP.2}
Let $\mu, \mu_1, \mu_2, \ldots$ be Radon measures on $\RR^n$. We say that
$\mu_i\to \mu$ or $\lim_{i\to\infty}\mu_i=\mu$ if 
\begin{enumerate}
\item[i)]$\limsup_{i\to\infty}F_K(\mu_i)<\infty$ for every compact set
$K\subset\RR^n$
\item[ii)]$\lim_{i\to\infty} F_K(\mu_i,\mu)=0$ for every compact set
$K\subset\RR^n$.
\end{enumerate}
\end{defn}

\begin{defn}\label{defn-GMTP.3}
Let $\mu, \mu_1, \mu_2, \ldots$ be Radon measures on $\RR^n$. We say that
$\{\mu_i\}$ converges weakly to $\mu$, $\mu_i\rightharpoonup\mu$ if
\begin{equation}\label{eqn-GMTP.3}
\lim_{r\to\infty}\int fd\mu:=\int fd\mu\ \ \ \forall\,\phi\in C_c(\RR^n).
\end{equation}
\end{defn}

\begin{lem}[\cite{P} Proposition 1.11]\label{lem-GMTP.1}
Let $\mu_1,\mu_2\ldots$ and $\mu$ be Radon measures on $\RR^n$ such that
$\limsup_{i\to\infty}\mu_i(K)<\infty$ for each compact set $K$ in $\RR^n$.
Then $\mu_i\to\mu$ if and only if $\mu_i\rightharpoonup \mu$.
\end{lem}

\begin{lem}[\cite{M}, Lemma 14.13]\label{lem-GMTP.2}
Let $\mu_1, \mu_2, \ldots$ and $\mu$ be Radon measures on $\RR^n$. Then
$\mu_i\to\mu$ if and only if
\begin{equation}\label{eqn-GMTP.4}
\lim_{i\to\infty} F_r(\mu_i,\mu)=0\ \ \ \forall\,r>0.
\end{equation}
\end{lem}

We now introduce a scale invariant relative of $F_r$, which behaves well
under weak convergence and scaling.

\begin{defn}[\cite{P}, \S2]\label{defn-GMTP.4}
\begin{enumerate}
\item[i)]A set $\cM$ of non-zero Radon measures in $\RR^n$ will be called
a cone if $c\Psi\in\cM$ whenever $\Psi\in\cM$ and $c>0$.
\item[ii)]A cone $\cM$ will be called a $d$-cone if $T_{0,r}[\Psi]\in\cM$
whenever $\Psi\in\cM$ and $r>0$.
\item[iii)]Let $\cM$ be a $d$-cone, and $\Phi$ a Radon measure in $\RR^n$
such that for $s>0$, $0<F_s(\Phi)<\infty$ then we define the {\emph{distance}}
between $\Phi$ and $\cM$ by 
\end{enumerate}
\begin{equation}\label{eqn-GMTP.5}
d_s(\Phi,\cM)=\inf\left\{F_s\left(\frac{\Phi}{F_s(\Phi)},\Psi\right):\Psi\in\cM\mbox{
and } F_s(\Psi)=1\right\}.
\end{equation}
\begin{enumerate}
\item[]We also define
\end{enumerate}
\begin{equation}\label{eqn-GMTP.6}
d_s(\Phi,\cM)=1\mbox{ if }F_s(\Phi)=0\mbox{ or }F_s(\Phi)=+\infty.
\end{equation}
\end{defn}

\begin{remark}\label{rem-GMTP.2}
Note that if $\cM$ is a $d$-cone and $\Phi$ is a Radon measure 
\begin{enumerate}
\item[i)]$d_s(\Phi,\cM)\le 1$
\item[ii)]$d_s(\Phi,\cM)=d_1(T_{0,s}[\phi], \cM)$
\item[iii)]if $\mu=\mathop{\lim}\limits_{i\to\infty}\mu_i$ and $F_s(\mu)>0$
then $d_s(\mu,\cM)=\mathop{\lim}\limits_{i\to\infty}d_s(\mu_i,\cM)$. 
\end{enumerate}
In fact if $\mu=\mathop{\lim}\limits_{1\i\to\infty}\mu_i$ then by Lemma
\ref{lem-GMTP.1} $\mu_i\rightharpoonup\mu$ and for $s>0$
\begin{equation}\label{eqn-GMTP.7}
F_s(\mu_i)=\int(s-|z|)^+d\mu_i\to\int(s-|z|)^+d\mu= F_s(\mu).
\end{equation}
Thus without loss of generality we may assume that $F_s(\mu_i)>0$ (at least
for $i$ large enough). Since ${\lim}_{i\to\infty}\mu_i=\mu$ then
$\mathop{\limsup}\limits_{i\to\infty} F_s(\mu_i)<\infty$ and therefore
$F_s(\mu)<\infty$. Let $\Psi\in\cM$ such that $F_s(\Psi)=1$ then
\begin{eqnarray}\label{eqn-GMTP.8}
F_s\left(\frac{\mu}{F_s(\mu)},\Psi\right) & \le &
F_s\left(\frac{\mu}{F_s(\mu)}, \frac{\mu}{F_s(\mu)}\right) +
F_s\left(\frac{\mu_i}{F_s(\mu)}, \frac{\mu_i}{F_s(\mu_i)}\right) + F_s
\left(\frac{\mu_i}{F_s(\mu_i)} , \Psi\right) \\
& \le & \frac{1}{F_s(\mu)} F_s(\mu, \mu_i) +
F_s(\mu_i)\left|\frac{1}{F_s(\mu)} - \frac{1}{F_s(\mu_i)}\right| +
F_s\left(\frac{\mu_i}{F_s(\mu_i)}, \Psi\right). \nonumber
\end{eqnarray}
Thus for any $\Psi\in\cM$ with $F_s(\Psi)=1$ we have
\begin{equation}\label{eqn-GMTP.9}
d_s(\mu,\cM) \le \frac{1}{F_s(\mu)} F_s(\mu,\mu_i) + F_s(\mu)
\left|\frac{1}{F_s(\mu)} - \frac{1}{F_s(\mu_i)}\right| +
F_s\left(\frac{\mu}{F_s(\mu_i)}, \Psi\right)
\end{equation}
which implies
\begin{equation}\label{eqn-GMTP.10}
d_s(\mu,\cM)\le \frac{F_s(\mu,\mu_i)}{F_s(\mu)} + F_s(\mu) \left|
\frac{1}{F_s(\mu)} - \frac{1}{F_s(\mu_i)} \right| + d_s(\mu_i,\cM).
\end{equation}
Letting $i\to\infty$ and combining (\ref{eqn-GMTP.4}) and
(\ref{eqn-GMTP.7}) we have that
\begin{equation}\label{eqn-GMTP.11}
d_s(\mu,\cM) \le \liminf_{i\to\infty} d_s(\mu_i, \cM).
\end{equation}
A similar calculation done reversing the roles of $\mu$ and $\mu_i$ yields
the inequality
\begin{equation}\label{eqn-GMTP.12}
\limsup_{i\to\infty} d_s(\mu_i,\cM)\le d_s(\mu, \cM)
\end{equation}
which proves the statement iii) in Remark \ref{rem-GMTP.2}.
\end{remark}

\begin{defn}\label{defn-GMTP.5}
\begin{enumerate}
\item[i)]Let $\eta$ be a Radon measure in $\RR^n$. Let $x\in\RR^n$, a
non-zero Radon measure $\nu$ in $\RR^n$ is said to be a tangent measure of
$\eta$ at $x$ if there are sequences $r_k\searrow 0$ and $c_k>0$ such that
$\nu=\mathop{\lim}\limits_{k\to\infty}c_kT_{x,r_k}[\eta]$.
\item[ii)]The set of all tangent measures to $\eta$ at $x$ is denoted by
$\Tan(\eta,x)$.
\end{enumerate}
\end{defn}

\begin{remark}\label{rem-GMTP.3}
For $\eta$ a non-zero Radon measure and $x\in\RR^n$, $\Tan(\eta,x)$ is a
$d$-cone. Moreover $\{\nu\in\Tan(\eta,x):F_1(\nu)=1\}$ is closed under weak
convergence (see \cite{P} 2.3).
\end{remark}

\begin{defn}\label{defn-GMTP.6}
The basis of a $d$-cone $\cM$ of Radon measures is the set
$\{\Psi\in\cM:F_1(\Psi)=1\}$. We say that $\cM$ has a closed (respectively
compact) basis, if its basis is closed (respectively compact) in the
topology induced by the metric
\[
\sum^\infty_{p=0}2^{-p}\min\{1, F_p(\Phi, \Psi)\}
\]
defined for Radon measures $\Psi$ and $\Phi$.
\end{defn}

\begin{prop}[\cite{P} Proposition 1.12]\label{prop-GMTP.0}
The set of Radon measures on $\RR^n$ with the metric above is a complete
separable metric space.
\end{prop}

\begin{remark}\label{rem-GMTP.4}
\begin{enumerate}
\item [i)]As indicated in \cite{P} 1.9(4), Proposition 1.12 and Proposition 1.11
the notion of convergence in this metric coincides with the notion of weak
convergence of Radon measures.
\item[ii)] A $d$-cone of Radon measures in $\RR^n$ has a closed basis if and
only if it is a relatively closed subset of the set of Radon measures in
$\RR^n$.
\end{enumerate}
\end{remark}

\begin{prop}[\cite{P} Proposition 2.2]\label{prop-GMTP.1}
Let $\cM$ be a $d$-cone of Radon measures. $\cM$ has a compact basis if
and only if for every $\lambda\ge 1$ there is $\tau>1$ such that $F_{\tau
r}(\Psi)\le \lambda F_r(\Psi)$ for every $\Psi\in\cM$ and every $r>0$.
In this case $0\in\spt\Phi$ for all $\Psi\in\cM$.
\end{prop}

The following theorem is in the same vein as Theorem 2.6 in \cite{P}.

\begin{thm}\label{thm-GMTP.1}
Let $\cF$ and $\cM$ be $d$-cones. Assume that  $\cF\subset\cM $, that $\cF$ is relatively closed with respect to the weak convergence of Radon measures and
that $\cM$ has a compact basis. Furthermore suppose that the following property holds:
\begin{equation*}
\left\{\begin{array}{l}
\exists\,\epsilon_0>0\mbox{ such that }\forall\,\epsilon\in(0,\epsilon_0)
\mbox{ there exists no }\mu\in\cM \mbox{ satisfying} \\
d_r(\mu,\cF)\le \epsilon\ \ \forall r\ge r_0>0\mbox{ and }
d_{r_0}(\mu,\cF)=\epsilon.\\
\end{array}\right.
\tag{P}
\end{equation*}
Then for a Radon measure $\eta$ and $x\in\spt\eta$ if
\begin{equation}\label{eqn-GMTP.13}
\Tan(\eta,x)\subset\cM\mbox{ and }\Tan(\eta,x)\cap\cF\ne\emptyset\ \mbox{ then }\Tan(\eta,x)\subset\cF.
\end{equation}
\end{thm}

\begin{cor}\label{cor-GMTP.1}
Let $\cF$ and $\cM$ be $d$-cones. Assume 
that  $\cF\subset\cM $, that $\cF$ is relatively closed with respect to the weak convergence of Radon measures and
that $\cM$ has a compact basis. Furthermore suppose that there exists
$\epsilon_0>0$ such that if $d_r(\mu,\cF)<\epsilon_0$ for all $r\ge r_0>0$,
then $\mu\in\cF$. Then for a Radon measure $\eta$ and $x\in\spt\eta$ if
\begin{equation}\label{eqn-GMTP.14}
\Tan(\eta,x)\subset\cM\mbox{ and }\Tan(\eta,x)\cap\cF\ne\emptyset\ \mbox{ then }\Tan(\eta,x)\subset\cF.
\end{equation}
\end{cor}

Note that the condition stated in Corollary \ref{cor-GMTP.1} is stronger
than condition (P) and a simple argument shows it.
\medskip

\noindent{\bf Proof of Theorem \ref{thm-GMTP.1}}: We proceed by
contradiction; i.e. assume that $\Tan(\eta,x)\subset\cM$,
$\Tan(\eta,x)\cap\cF\ne\emptyset$ but there exists
$\nu\in\Tan(\eta,x)\backslash\cF$. Since $\cF$ is closed there exists
$\epsilon_1\in(0,\frac{1}{2}\min\{\epsilon_0, 1\})$ such that
$d_1(\eta,\cF)>2\epsilon_1$. Moreover there exist $s_i\searrow 0$ and
$c_i>0$ such that $c_iT_{x,s_i}[\eta]\to\nu$. Since
$\Tan(\eta,x)\cap\cF\ne\emptyset$ there also exist $\delta_i>0$ and
$r_i\searrow 0$ such that $\delta_iT_{x,r_i}[\eta]\to\widetilde\nu\in\cF$.
Thus for $i$ large enough
\begin{equation}\label{eqn-GMTP.15}
d_1(T_{x,r_i}[\eta], \cF)  = d_1(\delta_i T_{x,r_i}[\eta],\cF)<\epsilon_1, \  \hbox{ and }\  
d_1(T_{x,s_i}[\eta],\cF)  >  \epsilon_1
\end{equation}
Without loss of generality we may assume that $s_i<r_i$. Let
$\tau_i\in\left(\frac{s_i}{r_i},1\right)$ be the largest number such that
$\tau_ir_i=\rho_i$ satisfies
\begin{equation}\label{eqn-GMTP.16}
d_1(T_{x,\rho_i}[\eta],\cF)=\epsilon_1.
\end{equation}
Hence for all $\alpha\in\left({\tau_i}, 1\right)$
\begin{equation}\label{eqn-GMTP.17}
d_1(T_{x,\alpha r_i}[\eta],\cF)=d_{\alpha/\tau_i}(T_{x,\rho_i}[\eta],\cF)<\epsilon_1.
\end{equation}
We claim that $\tau_i\to 0$ as $i\to\infty$. In fact, otherwise there
exists a subsequence
$\tau_{i_k}\to\tau\in(0,1)$, and
$\delta_{i_k}T_{x,\rho_{i_k}}[\eta]=\delta_{i_k}T_{x,\tau_{i_k}r_{i_k}}[\eta]\to
T_{0,\tau}[\widetilde\nu]\in\cF$, which implies that
$d_1(T_{x,\rho_{i_k}}[\eta],\cF)\to 0$ as $i_k\to\infty$ which contradicts
(\ref{eqn-GMTP.16}). Therefore (\ref{eqn-GMTP.16}) and (\ref{eqn-GMTP.17})
yield
\begin{equation}\label{eqn-GMTP.18}
\lim_{i\to\infty}d_1(T_{x,\rho_i}[\eta],\cF)=\epsilon_1
\end{equation}
and for every $r>1$,
\begin{equation}\label{eqn-GMTP.19}
\limsup_{i\to\infty} d_r(T_{x,\rho_i}[\eta],\cF)\le \epsilon_1.
\end{equation}
Note that $F_r(T_{x,\rho_i}[\eta])=\frac{1}{\rho_i}
F_{B(x,r\rho_i)}(\eta)\in(0,\infty)$ for $x\in\spt\eta$. Moreover a simple
calculation shows that for $i$ large enough 
\begin{equation}\label{eqn-GMTP.20}
0<\frac{r}{2}\eta\left(B\left(x,\frac{r\rho_i}{2}\right)\right)\le
F_r(T_{x,\rho_i}[\eta])\le r\eta(B(x,r\rho_i)) \le r\eta(B(x,r))<\infty.
\end{equation}
Since $\epsilon_1<1$, $\lambda=\frac{2}{1+\epsilon_1}>1$ and by Proposition
\ref{prop-GMTP.1} there is $\tau>1$ so that $F_{\tau r}(\Psi)\le\lambda
F_r(\Psi)$ for every $\Psi\in\cM$ and every $r>0$. For $r\ge 1$ and $i$
large enough there is $\Psi\in\cM$ so that $F_{\tau r}(\Psi)=1$ and
\begin{equation}\label{eqn-GMTP.21}
F_r\left(\frac{T_{x,\rho_i}[\eta]}{F_{r\tau}(T_{x,\rho_i}[\eta])},\Psi\right)
\le F_{\tau r}
\left(\frac{T_{x,\rho_i}[\eta]}{F_{r\tau}(T_{x,\rho_i}[\eta])},
\Psi\right)\le\epsilon_1.
\end{equation}
Hence
\begin{equation}\label{eqn-GMTP.22}
\frac{F_r(T_{x,\rho_i}[\eta])}{F_{\tau r}(T_{x,\rho_i}[\eta])} \ge
F_r(\Psi)-\epsilon_1\ge \frac{1+\epsilon_1}{2} F_{\tau r}(\Psi)-\epsilon_1 =
\frac{1-\epsilon_1}{2}.
\end{equation}
Thus for $p=1,2,\ldots$ (\ref{eqn-GMTP.22}) yields
\begin{equation}\label{eqn-GMTP.23}
\limsup_{i\to\infty}
\frac{F_{\tau^p}(T_{x,\rho_i}[\eta])}{F_1(T_{x,\rho_i}[\eta])} \le
\left(\frac{1-\epsilon_1}{2}\right)^{-p}.
\end{equation}

Combining (\ref{eqn-GMTP.20}), (\ref{eqn-GMTP.23}) and i) in Remark
\ref{rem-GMTP.1} we conclude that for $p=1,2, \cdots$, $\tau>1$ (as above),
and $i$ large enough
\begin{equation}\label{eqn-GMTP.24}
\frac{T_{x,\rho_i}[\eta](B(0,\tau^p))}{F_1(T_{x,\rho_i}[\eta])} \le 
2\left(\frac{2}{1-\epsilon_1}\right)^p\tau^{-p} 
 \le  2\left(\frac{2}{(1-\epsilon_1)\tau}\right)^p. 
\end{equation}
Thus for any $s>0$, (\ref{eqn-GMTP.24}) ensures that
\begin{equation}\label{eqn-GMTP.25}
\limsup_{i\to\infty}
\frac{T_{x,\rho_i}[\eta](B(0,s))}{F_1(T_{x,\rho_i}[\eta])}<\infty.
\end{equation}
By the compactness theorem for Radon measures there exists a subsequence
$i_k$ such that $\frac{T_{x,\rho_{i_k}}[\eta]}{F_1(T_{x,\rho_i}[\eta])}$
converges to a Radon measure $\Phi\in\cM$ (as $\cM$ has a closed basis),
satisfying $F_1(\Phi)=1$. Therefore $F_r(\Phi)>0$ for $r\ge 1$.

Combining iii) in Remark \ref{rem-GMTP.2} with (\ref{eqn-GMTP.18}) and
(\ref{eqn-GMTP.19}) we have that
\begin{equation}\label{eqn-GMTP.26}
d_1(\Phi,\cF)=\epsilon_1
\end{equation}
and
\begin{equation}\label{eqn-GMTP.27}
d_r(\Phi,\cF)\le \epsilon_1\mbox{ for all }r>1.
\end{equation}
Since $\epsilon_1<\epsilon_0/2$ (\ref{eqn-GMTP.26}) and (\ref{eqn-GMTP.24})
contradict condition (P). This concludes the proof of Theorem
\ref{thm-GMTP.1}.\qed

We next recall a couple of results from \cite{P} and \cite{M}. They provide
additional information about $\Tan(\Phi,x)$ for a Radon measure $\Phi$ and
$x\in\spt\Phi$. The first result yields conditions that ensure that
$\Tan(\Phi,x)$ has a compact basis. As we will see these conditions are
satisfied by the harmonic measures considered in this paper. The second
result states that tangent measures to tangent measures of $\Phi$ are
tangent measures of $\Phi$.

\begin{thm}[\cite{P} Corollary 2.7]\label{thm-GMTP.2}
Let $\Phi$ be a Radon measure in $\RR^n$, and $x\in\spt\Phi$.
$\Tan(\Phi,x)$ has a compact basis if and only if
\begin{equation}\label{eqn-GMTP-doubling}
\mathop{\limsup}\limits_{r\to 0}
\frac{\Phi(B(x,2r))}{\Phi(B(x,r))}<\infty.
\end{equation}
\end{thm}

\begin{thm}[\cite{M}, Theorem 14.16]\label{thm-GMTP.3}
Let $\Phi$ be a Radon measure in $\RR^n$, $\Phi$ a.e. $a\in\RR^n$, if
$\Psi\in\Tan(\Phi,a)$ then
\begin{enumerate}
\item[i)]$T_{x,\rho}[\Psi]\in\Tan(\Phi,a)$ for all $x\in\spt\Psi$ and all
$\rho>0$
\item[ii)]$\Tan(\Psi,x)\subset\Tan(\Phi,a)$ for all $x\in\spt\Psi$.
\end{enumerate}
\end{thm}

Finally we present a couple of results which will be used later in the
paper.

\begin{defn}\label{defn-GMTP.7}
A Radon measure $\o$ in $\RR^n$ is said to be locally doubling if for every
compact set $K\subset\spt \o$ there exists $C=C_k\ge 1$ and $R_K=R>0$ such
that for $Q\in K$, and $s\in(0,R)$
\begin{equation}\label{eqn-GMTP.28}
\o(B(Q,2s))\le C\o(B(Q,s))
\end{equation}
\end{defn}

\begin{lem}\label{lem-GMTP.3}
Let $\o$ be a locally doubling measure in $\RR^n$. Let $\Psi$ be a non-zero
Radon measure with $\Psi\in\Tan(\o,Q)$. There exists a sequence of positive
numbers $\{r_i\}_{i\ge 1}$ with $\mathop{\lim}\limits_{i\to\infty}r_i=0$
such that ${r_i}^{-1}(\spt\o-Q)$ converges to $\spt\Psi$ in the
Hausdorff distance sense uniformly on compact sets.
\end{lem}

\begin{proof}
Since $\Psi\in\Tan(\o,Q)$, and $\o$ is locally doubling by Remark (3) in
14.4 \cite{M} we have that there are a sequence $r_i\downarrow 0$ and a
positive constant $c$ such that $\Psi=c\mathop{\lim}\limits_{i\to\infty}\o
(B(Q,r_i))^{-1}T_{Q,r_i}[\o]$. Let $X=\mathop{\lim}\limits_{i\to\infty}
X_i\in B(0,R_0)$ where $X_i=r^{-1}_i(Z_i-Q)$ with $Z_i\in\spt\o$. For $r\in
(0,1)$ there exists $i_0\ge 1$ such that for $i\ge i_0$,
$|X-X_i|<\frac{r}{2}$ and $|Z_i-Q|\le r_i|X_i|\le r_i(|X|+1)\le R_0+1$.
Since $\o$ is locally doubling there exists $C_0\ge 1$ and $R>0$ such that
for $P\in B(Q,2(R_0+1))$ and $s<R$, $\o(B(P,2s))\le C_0\o(B(P,s))$. Thus for
$r\le \min\{R,1\}$ and $i$ large enough so that $r_i(R_0+1)<R$ we have

\begin{eqnarray}\label{eqn-GMTP.29}
\frac{T_{Q,r_i}[\o](B(x,r))}{\o(B(Q,r_i))} & = & \frac{\o(B(Q+r_iX,
rr_i))} {\o(B(Q,r_i))} \\
& \ge & 
\frac{\o(B(Q+r_iX_i, r_i(r-|X-X_i|))}{\o (B(Q,r_i))} \nonumber \\
& \ge & 
\frac{\o(B(Z_i, \frac{rr_i}{2}))}{\o(B(Z_i,r_i(R_0+2)))}\ge C_0^k>0. \nonumber 
\end{eqnarray}
where $k\in\NN$ is such that $2^{-k}(R_0+2)\le \frac{r}{2}<2^{-(k-1)}(R_0+2)$.
Thus
\begin{eqnarray}\label{eqn-GMTP.30}
\Psi(B(x,2r)) & \ge & \Psi(\overline{B(x,r)}) \\
& \ge & \limsup_{i\to\infty} c
\frac{T_{a,r_i}[\o](\overline{B(x,r))}}{\o(B(Q,r_i))}\ge C_0^k>0.
\nonumber
\end{eqnarray}
Thus (\ref{eqn-GMTP.30}) ensures that $x\in\spt\Psi$.
This shows that $\mathop{\lim}\limits_{i\to\infty} r^{-1}_i
(\spt\o-Q)\subset\spt\Psi$. To show the opposite inclusion assume that
$X\not\in\mathop{\lim}\limits_{i\to\infty} r^{-1}_i(\spt\o-Q)$ there exists
$\{r_{i_k}\}\subset\{r_i\}$, $r_{i_k}\searrow 0$ such that
$d(X,{r^{-1}_{i_k}}(\spt\o-Q))\ge \epsilon_0$. Thus $B(x,\frac{\epsilon_0}{2})\cap
{r^{-1}_{i_k}}(\spt\o-Q)=\emptyset$. For $\varphi\in
C^\infty_c(B(X,\frac{\epsilon_0}{2}))$ we have
\begin{equation}\label{eqn-GMTP.31}
\int\varphi d\Psi = C\lim_{i_k\to\infty} \frac{1}{\o(B(Q,r_{i_k}))}
\int\varphi\left(\frac{Y-Q}{r_{i_k}}\right)d\o=0,
\end{equation}
which ensures that $X\not\in\spt\Psi$.
\qed
\end{proof}

The following lemma is a simple geometric measure theory fact which allows
us to give an estimate on the Hausdorff dimension of sets which approach $(n-1)$-planes
locally.

\begin{lem}\label{lem-size.2}
Let $\Sigma\subset \RR^n$ be such that $\forall Q\in\Sigma$
\begin{equation}\label{poitwise-flat}
\lim_{r\to 0}\beta_{\Sigma}(Q,r)=0\  \hbox{ where }\ 
\beta_{\Sigma}(Q,r)=\inf_{L\in G(n,n-1)}\sup_{y\in B(Q,r)\cap \Sigma}\frac{d(y,L)}{r}.
\end{equation}
Then
\begin{equation}\label{eqn-size.8}
\dim_\H\Sigma\le n-1.
\end{equation}
\end{lem}

The following proof is an adaptation of the argument used in \cite{Si} to
prove Lemma 3 in Chapter 3, \S 4.

\begin{proof}
Let $Q\in\Sigma$. Given $\e>0$ there exists $r_{Q,\e}>0$ such that for 
$r<r_{Q,\e}$ there exists an $(n-1)$ plane $L(Q,r)$
through $Q$ so that
\begin{equation}\label{eqn-size.9}
\Sigma\cap B(Q,r)\subset (L(Q,r)\cap B(Q,r):\e r).
\end{equation}
Note that for $\e>0$
\begin{equation}\label{eqn-size.10}
\Sigma=\bigcup^{\infty}_{j=1}\Sigma_j\mbox{ where
}\Sigma_j=\{Q\in\Sigma:r_{Q,\e}>\frac{10}{2^j}\}.
\end{equation}
Without loss of generality we may assume that $0\in\Sigma$. Let $k\in\NN$. For
$j_0\ge 1$ cover $\Sigma_{j_0}\cap B(0,k)$ by sets $\{C_s\}_{s\ge 1}$ of
diameter less than $\d>0$. Choosing $\d<\frac{1}{2^{j_0}}$ we can ensure
that each such set is contained in a ball of center $Q\in\Sigma_{j_0}$ and
radius $r_Q=\diam C_s$ for some $s$ with $Q\in C_s$ less than $\d$, i.e.
$r_Q<\frac{1}{10}r_{Q,\e}$. Note that  $B(Q,r_Q)\cap L(Q,r_Q)$ can be
covered by $N\e^{-n+1}$ balls $\{B_l\}_l$ centered in $L(Q,r_Q)$ with
radius $5\e r_Q$ and such that the balls of same center and radius $\e r_Q$
are disjoint. Here $N>0$ is an absolute constant that only depends on $n$. Thus
for $\g>0$
\begin{equation}\label{eqn-size.11}
\sum_l(\diam\, B_l)^{n-1+\g}  =  (5\e r_Q)^{n-1+\g} N\e^{-n+1} 
 =  5^{n-1}\e^\g N r^{n-1+\g}_Q. 
\end{equation}
Note that if $\g\ge -\frac{\ln (4N5^{n-1})}{\ln (5\e)}$ then 
\begin{equation}\label{eqn-size.12}
\sum_l(\diam B_o)^{n-1+\g} \le \frac{1}{4} r_Q^{n-1+\g}.
\end{equation}
Thus for $\d<\frac{1}{2^{j_0}}$
\begin{equation}\label{eqn-size.13}
\H^{n-1+\g}_{5\e\d} (\Sigma_{j_0} \cap B(0,k)) \le \frac{1}{4}
\H^{n-1+\g}_\d (\Sigma_{j_0} \cap B(0,k)).
\end{equation}
Letting $\d\to 0$ we conclude that $\forall\,j\in\NN$, and $\g\ge -\ln(4N5^{n-1})/\ln (5\e)$ 
\begin{equation}\label{eqn-size.14}
\H^{n-1+\g} (\Sigma_j\cap B(0,k))=0.
\end{equation}
Thus (\ref{eqn-size.10}) ensures that
\begin{equation}\label{eqn-size.15}
\H^{n-1+\g} (\Sigma\cap B(0,k))=0
\end{equation}
for $\g\ge -\frac{\ln (4N5^{n-1})}{\ln (5\e)}$ and $\e>0$. Letting $\e\to 0$ we
conclude that
$\H^{n-1+\g}(\Sigma\cap B(0,k))=0$,for all $\g>0$ and hence
$\H^{n-1+\g}(\Sigma)=0$ also. This implies that
$\dim_\H\Sigma\le n-1$.\qed
\end{proof}

\section{Two sided locally non-tangentially accessible domains}

\begin{defn}\label{prel-defn1}
A domain $\O\subset\RR^n$ is admissible if 
\begin{itemize}
\item $\O^+=\O$ and 
$\O^-={\rm int}\,\O^c$ are regular for the Dirichlet problem.
\item $\po^+=\po^-=\po$.
\item There exist points $X^\pm\in\O^\pm$
such that for every point $Q\in\del\O$ there exists 
$0<R<\min\{\delta(X^+),\delta(X^-)\}$ satisfying 
$u\in C^{0}(B(Q,R))\cap \H^1(B(Q,R))$,
where $\delta(X)={\rm{dist}}\, (X,\del\O))$ and 
\begin{equation}\label{prel-eqn1}
u(X)=G_+(X,X^+)-G_-(X,X^-)
\end{equation}
and $G_\pm(-,X^\pm)$ denote the Green function of $\O^\pm$ with pole at
$X^\pm$.
\end{itemize}
\end{defn}

{\bf Notation:} If $\O$ is admissible so is ${\rm int}\, \O^c$.
Let $\O$ be an admissible domain we denote by $\omega^\pm$ the
harmonic measure of $\O^\pm$ with pole $X^\pm$. Note that in this case
$u^\pm=G_\pm(-,X^\pm)$.

The monotonicity formula of Alt, Caffarelli and Freidman plays a 
role in this work. We recall several of the results which will be used
later.

\begin{thm}{\bf{\cite{ACF}}}\label{prel-acf1}
Let $\O$ be an admissible domain. Then $Q\in\del\O$ there exists 
$0<R<\min\{\delta(X^+),\delta(X^-)\}$ such that the quantity
\begin{equation}\label{prel-gamma-1}
\gamma(Q,r)=\left(\frac{1}{r^2}\int_{B(Q,r)}\frac{|\nabla u^+|^2}{|X-Q|^{n-2}}\, dX \right)
\cdot \left(\frac{1}{r^2}\int_{B(Q,r)}\frac{|\nabla u^-|^2}{|X-Q|^{n-2}}\, dX \right)
\end{equation}
is an increasing function of $r$ for $r\in (0,R)$ and $\gamma(Q,R)<\infty$.
\end{thm}

Note that the ACF-monotonicity formula ensures that
\begin{equation}\label{prel-gamma2}
\gamma(Q)=\lim_{r\rightarrow 0}\gamma(Q,r)
\end{equation}
exists and it is a non-negative finite quantity. A combination of the results of 
Alt-Caffarelli- Friedman, Beckner-Kenig-Pipher and Brothers-Ziemer
asserts that if $\gamma(Q)>0$ then all blow-ups of the boundary at $Q$
are $(n-1)$-planes (see \cite{ACF}, \cite{BKP} and \cite{BZ}). This last fact will not be used here.

Our immediate goal is to estimate $\gamma(Q,r)$ in terms of $\omega^\pm$ and
$u^\pm$. Let $\varphi\in C^\infty_c(\RR^n)$. The harmonic extension $v_\varphi$ 
of 
$\varphi$ to $\O$ (i.e $\Delta v_\varphi=0$ in $\O$ and 
$v_\varphi=\varphi$ in $\del\O$)
satisfies
\beq\label{prel-eqn4}
v_\varphi(Y)=-\int \lan\nabla G(Y,X),\nabla\varphi(X)\ran\, dX
\eeq
Let $R<\min\{\delta(X^+),\delta(X^-)\}$, $2r<R$, $Q\in\del\O$ 
and $\varphi=1$ on $B(Q,\frac{3r}{2})$, $\varphi=0$ on $B(Q, 2r)^c$,
$0\le \varphi \le 1$ and $|\nabla \varphi|\le \frac{C}{r}$. By the maximum 
principle $v^\pm_\varphi(X^\pm)\ge \omega^\pm(B(Q,r))$. Here  $v^\pm_\varphi$
denotes the harmonic extension of $\varphi$ to $\O^\pm$. Hence by 
(\ref{prel-eqn4}) we have
\bea\label{prel-eqn5}
\omega^\pm(B(Q,r)) &\le&  \frac{C}{r}\int_{B(Q,2r)\backslash B(Q,r)}
|\nabla u^\pm|\\
 &\le &\frac{C}{r}\left(\int_{B(Q,2r)\backslash B(Q,r)}
\frac{|\nabla u^\pm|^2}{|X-Q|^{n-2}}\right)^{1/2}
\left(\int_{B(Q,2r)\backslash B(Q,r)}|X-Q|^{n-2}\right)^{1/2},\nonumber
\eea
which yields
\beq\label{prel-eqn6}
\frac{\omega^\pm(B(Q,r))}{r^{n-1}}\le C\left(\frac{1}{r^2}
\int_{B(Q,2r)\backslash B(Q,r)}
\frac{|\nabla u^\pm|^2}{|X-Q|^{n-2}}\right)^{1/2}
\eeq
and
\beq\label{prel-eqn7}
\frac{\omega^+(B(Q,r))}{r^{n-1}}\cdot \frac{\omega^-(B(Q,r))}{r^{n-1}}\le C
\gamma(Q,2r)^{1/2}.
\eeq

Note that $\Delta (u^\pm)^2= 2|\nabla u^\pm|^2\ge 0$ because $u^\pm$ is zero
on the support of the measure $\Delta u^\pm$. 
Using Cacciopoli's inequality as well as the fact that $(u^\pm)^2$ is 
subharmonic (and therefore the averages over spheres are increasing 
as a function of the radius) we have for $Q\in\po$ 
that
\bea\label{prel-eqn8}
\int_{B(Q,r)}\frac{|\nabla u^\pm|^2}{|X-Q|^{n-2}}&=&\frac{1}{2}
\int_{B(Q,r)}\frac{\Delta(u^\pm)^2}{|X-Q|^{n-2}}\\
&=& (u^\pm)^2(Q) + \frac{1}{r^{n-2}}
\int_{\del B(Q,r)}u^\pm\frac{\del u^\pm}{\del r} +\frac{n-2}{2r^{n-1}}
\int_{\del B(Q,r)}(u^\pm)^2\nonumber\\
&= & \frac{1}{r^{n-2}}\int_{B(Q,r)}|\nabla u^\pm|^2 +
\frac{n-2}{2r^{n-1}}
\int_{\del B(Q,r)}(u^\pm)^2\nonumber\\
&\le & C \frac{1}{r^n}\int_{B(Q,2r)} (u^\pm)^2 +\frac{n-2}{2r^{n-1}}
\int_{\del B(Q,r)}(u^\pm)^2\nonumber\\
&\le & C \frac{1}{r^n}\int_{B(Q,2r)} (u^\pm)^2 +\frac{n-2}{2r^{n}}
\int_r^{2r}\left(\int_{\del B(Q,s)}(u^\pm)^2\right)\, ds\nonumber\\
&\le & C \frac{1}{r^n}\int_{B(Q,2r)} (u^\pm)^2\nonumber
\eea

We have proved the following result:

\begin{lem}\label{prel-lem1}
Given $\O\subset \RR^n$ an admissible domain. 
Let $R<\min\{\delta(X^+),\delta(X^-)\}$, $4r<R$, $Q\in\del\O$, then
\beq\label{prel-eqn9}
\frac{\omega^\pm(B(Q,r))}{r^{n-1}}\le 
C\left(\frac{1}{r^2}\int_{B(Q,2r)}
\frac{|\nabla u^\pm|^2}{|X-Q|^{n-2}}\right)^{1/2}
\le C\left(\frac{1}{r^{n+2}}\int_{B(Q,4r)}(u^\pm)^2\right)^{1/2}
\eeq
Therefore
\bea\label{prel-eqn10}
\frac{\omega^+(B(Q,r))}{r^{n-1}}\cdot \frac{\omega^-(B(Q,r))}{r^{n-1}}&\le& C
\gamma(Q,2r)^{1/2}\\
\gamma(Q,2r)&\le & C
\left(\frac{1}{r^{n+2}}\int_{B(Q,4r)}(u^+)^2\right)\cdot
\left(\frac{1}{r^{n+2}}\int_{B(Q,4r)}(u^-)^2\right)\nonumber
\eea
\end{lem}

\begin{remark}
For $Q_0\in\po$, $r_0<\frac{R}{8}$ and $Q\in 
B\left(Q_0,\frac{r_0}{2}\right)\cap\po$, we have for $r<\frac{r_0}{2}$,
\begin{eqnarray}\label{prel-eqn13A}
\gamma(Q,r) & \le & \gamma\left(Q,\frac{r_0}{2}\right) \\
& \le & C\left(\frac{1}{r^{n+2}_0}\int_{B(Q,r_0)}(u^+)^2\right) \cdot
\left(\frac{1}{r_0^{n+2}} \int_{B(Q,r_0)}(u^-)^2 \right)\nonumber
\end{eqnarray}
Moreover
\begin{equation}\label{prel-eqn13B}
\frac{\o^+ (B(Q,r))}{r^{n-1}} \cdot \frac{\o^-(B(Q,r))}{r^{n-1}}  \le 
C\gamma(Q,2r)^{\frac{1}{2}} \le
\frac{C}{r^{n+2}_0}\|u\|^2_{L^2(B(Q_0,4r_0))}.
\end{equation}
Here $C$ only depends on $n$. Thus Beurling's inequality (see \cite{GM} Chapter IV, Theorem 6.2 and 
Chapter VI, proof of Theorem 6.3) holds in
higher dimensions. 
\end{remark}

\begin{defn}\label{prel-defn600}
A domain $\O\subset\RR^{n}$ satisfies the 
\emph{corkscrew condition}
if for each compact set $K\subset\RR^{n}$ there exists $R>0$
such that for $Q\in\po\cap K$ and $r\in(0,R]$ there exists
there exists $A=A(Q,r)\in\O$ such that $M^{-1}r<|A-Q|<r$ and
$d(A,\po)>M^{-1}r$.  If $\O$ is unbounded we require that $R=\infty$.
\end{defn}

\begin{defn}\label{defn1.2}
A domain $\O\subset\RR^n$ is locally non-tangentially accessible (NTA) if
\begin{enumerate}
\item $\O^\pm$ satisfy the corkscrew condition.
\item \emph{Harnack Chain Condition}. 
Given a compact set $K\subset\RR^n$ there exists $R=R_K>0$ and $M=M_K>1$
such that if $\e>0$, and $X_1,
X_2\in\O\cap B(Q,\frac{r}{4})$ for some $Q\in\po\cap K$, $r<R$, 
$d(X_j,\po)>\e$ and $|X_1-X_2|<2^k\e$, then there
exists a Harnack chain from $X_1$ to $X_2$ of length $Mk$ and such that the
diameter of each ball is bounded below by $M^{-1}\min\{\dist(X_1,\po),
\dist(X_2,\po)\}$.  If $\O$ is unbounded we require that $R=\infty$.
\end{enumerate}
\end{defn}

If $\O$ is bounded and locally NTA then $\O$ is NTA as defined in
\cite{JK}.

In particular since most of the results concerning the behaviour of non
negative harmonic measures on NTA domains are local, suitable modifications
hold for locally NTA domains. We briefly summarize the most important ones
in the current context.

\begin{lem}[\cite{JK}, Lemma 4.1]\label{lem1.2}
Let $\O$ be a locally NTA domain. Given a compact set $K\subset\RR^n$ there
exists $\beta>0$ such that for all $Q\in\po\cap K$, $0<2r<R_K$, and every
positive harmonic function $u$ in $\O\cap B(Q, 2r)$, which vanishes
continuously on $B(Q,2r)$ then for $X\in\O\cap B(Q,r)$
\begin{equation}\label{eqn1.16}
u(X)\le C\left(\frac{|X-Q|}{r}\right)^\beta\sup\{u(Y):Y\in\p B(Q,2r)\cap \O\}.
\end{equation}
here $C$ only depends on $K$.
\end{lem}

\begin{lem}[\cite{JK}, Lemma 4.4]\label{lem1.3}
Let $\O$ be a locally NTA domain. Given a compact set $K\subset\RR^n$ for
$Q\in \po\cap K$ and $0<2r<R_K$. If $u\ge 0$ is a harmonic function in
$\O\cap B(Q,4r)$ and $u$ vanishes continuously on $B(Q,2r)\cap \po$ then
\begin{equation}\label{eqn1.17}
u(Y)\le Cu(A(Q,r)),
\end{equation}
for all $Y\in B(
Q,r)\cap\O$. Here $C$ only depends on $K$.
\end{lem}

\begin{lem}[\cite{JK}, Lemma 4.8]\label{lem1.4}
Let $\O$ be a locally NTA domain. Given a compact set $K\subset\RR^n$ for
$Q\in\po\cap K$, $0<2r<R_K$ and $X\in\O\backslash B(Q,2r)$. Then
\begin{equation}\label{eqn1.18}
C^{-1}<\frac{\o^X(B(Q,r))}{r^{n-2}G(A(Q,r),X)} < C,
\end{equation}
where $G(A(Q,r),X)$ is the Green function of $\O$ with pole $X$.
\end{lem}

\begin{lem}[\cite{JK}, Lemma 4.8, 4.11]\label{lem1.5}
Let $\O$ be a locally NTA domain. Given a compact set $K\subset\RR^n$ if
$M>1$ and $R>0$, are as in Definition \ref{defn1.2}, for $Q\in\po\cap K$,
$0<2r<R$, and $X\in\O\backslash B(Q,2Mr)$, then for $s\in [0,r]$ 
\begin{equation}\label{eqn1.19}
\o^X(B(Q,2s))\le C\o^X(B(Q,s)),
\end{equation}
where $C\ge 1$ only depends on $K$.
\end{lem}

\begin{defn}\label{2-nta}
A domain $\Omega\subset\RR^n$ is 2-sided locally non-tangentially accessible 
if $\Omega^\pm$ are both locally NTA.
\end{defn}

\begin{lem}\label{lem-nta-ad} Let $\Omega\subset\RR^n$ be a 2-sided locally
NTA domain, then $\Omega$ is an admissible domain.
\end{lem}

\begin{proof} Lemmas \ref{lem1.2} and \ref{lem1.3}  ensure that there exists $M>1$ depending on the NTA constants of $\O^\pm$ such that for $X^\pm\in\O^\pm$, and for $r<\frac{1}{M}\min\{\delta(X^+),\delta(X^-)\}$ if
$X\in \O^\pm\cap B(Q,r)$ then
\begin{equation}\label{holder-c}
G_\pm(X,X^\pm)\le C G_\pm(A^\pm(Q,r), X^\pm)\left(\frac{\delta(X)}{R}\right)^\beta,
\end{equation}
where $\beta$ and $C$ depend on $n$ and the NTA constants of $\O^\pm$.

Thus $u=G_+(-,X^+)-G_-(-,X^-)\in C^0(B(Q,R))$. Recall that for $X\in
B(Q,r)\cap\O^{\pm}$
\begin{equation}\label{eqn-NTA.2}
|\nabla G_\pm(X,X^\pm)|\le C_n \frac{G_\pm(X,X^\pm)}{\d(X)}.
\end{equation}
We claim that there exist $\eta>0$ and $R\in (0,R_0)$ so that
\begin{equation}\label{eqn-NTA.2A}
\int_{B(Q,r)} \left(\frac{G_\pm(X,X^\pm)}{\d(X)}\right)^{2+\eta} dX <
\infty\qquad\mbox{ for }r<R.
\end{equation}
Note that 
\begin{equation}\label{eqn-NTA.3}
\int_{B(Q,r)}\left(\frac{G_\pm(X,X^\pm)}{\d(X)}\right)^{2+\eta}dX =
\sum^\infty_{j=0} \int_{\{2^{-s-1}r\le
\d(X)<2^{-j}r\}}\left(\frac{G_\pm(X,X^\pm)}{\d(X)}\right)^{2+\eta}
\end{equation}
cover $\{X\in B(Q,r); 2^{-j-1}r\le\d(X)<2^{-j}r\}\cap \O^\pm=A^\pm_j$ by
balls $\{B^\pm\left(X^j_i, \frac{r}{2^{j-2}}\right)\}^{N_j}_{i=1}$ such
that $X^j_i\in A^\pm_j$, $|X^j_i-X^j_l|\ge \frac{r}{2^{j-2}}$ for $i\ne l$.
These balls have finite overlaps bounded by a number which only depends on
$n$. Moreover $\dist\left(B^\pm\left(X^j_i, \frac{r}{2^{j-2}}\right)\cap
A^+_j,\, \po\right)\ge \frac{r}{2^{j+2}}$. Note that for $X\in A^+_j$
(\ref{eqn1.16}) yields
\begin{equation}\label{eqn-NTA.4}
G_+(X,X^+)\le C G_+(A^+(Q,r); X^+)2^{-j\beta}
\end{equation}
and
\begin{eqnarray}\label{eqn-NTA.5}
\int_{2^{-j-1}r\le \d(X)<2^{-j}r}
\left(\frac{G(X,X^+)}{\d(X)}\right)^{2+\eta}\, dX & \le & C r^{-(1+\eta)} 2^{j(1+\eta)}
2^{-j\beta(1+\eta)} G_+(A^+(Q,r), X^+)^{1+\eta}\nonumber \\
 &&\qquad \cdot\int_{2^{-j-1}r\le \d(X) <
2^{-j}r} \frac{G(X,X^+)}{\d(X)}dX.
\end{eqnarray}

For $X\in B^+\left(X^{j}_i, \frac{r}{2^{j+2}}\right)$ \cite[Lemma 4.8]{JK}
yields
\begin{equation}\label{eqn-NTA.6}
\frac{G_+(X,X^+)}{\d(X)} \sim
\frac{\o^{+}\left(B(Q_X,\d(X)\right)}{\d(X)^{n-1}}
\end{equation}
where $Q_X\in\po$ is such that $|X-Q_X|=\d(X)$.
The notation $a\sim b$ means that there exists a constant, $C>1$ such that
$C^{-1}\le a/b\le c$.  By Harnack's principle for
$X\in B^+\left(X^j_i, \frac{r}{2^{j-2}}\right)\cap A^+_j$
\begin{equation}\label{eqn-NTA.7}
G_+(X,X^+)\sim G_+(X^j_i, X^+).
\end{equation}
Note also that for $X\in B^+\left(X^j_i, \frac{r}{2^{j-2}}\right)\cap
A^+_j$, $\d(X)\sim\frac{r}{2^j}\sim\d(X^j_i)$. Combining this remark with
the doubling property of $\o^{\pm}$ (see \cite{JK},  4.9 \& 4.11), 
(\ref{eqn-NTA.6}) and (\ref{eqn-NTA.7}) we obtain that for $X\in B^+
\left(X^j_i, \frac{r}{2^{j-2}}\right)\cap A^+_j$
\begin{equation}\label{eqn-NTA.8}
\frac{G_+(X,X^+)}{\d(X)} \sim \frac{G_+(X^j_i,X^+)}{\d(X^j_i)}\sim
\frac{\o^{+}\left(B\left(Q^j_i,
\frac{r}{2^j}\right)\right)}{\left(\frac{r}{2^i}\right)^{n-1}}
\end{equation}
where $Q^j_i\in\po$ is such that $\d(X^j_i)=|X^j_i-Q^j_i|$. In particular
\begin{eqnarray*}
|Q^j_i-Q^j_l| & \ge & |X^j_i-X^j_l|-|X^j_i-Q^j_i|-|X^j_l-Q^j_l| \\
& \ge & |X^j_i-X^j_l|-\frac{r}{2^j}\ge \frac{r}{2^{j-2}}-\frac{r}{2^j}\ge
\frac{r}{2^j}
\end{eqnarray*}
Thus $\left\{B\left(Q^j_i; \frac{r}{2^j}\right)\right\}^{N_j}_{i=1}$ is a
disjoint family of balls in $B(Q,2r)$. Hence the doubling property
of $\o^{+}$ and (\ref{eqn-NTA.8}) yield
\begin{eqnarray} \label{eqn-NTA.9}
\int_{2^{-j-1}\le \d(X)<2^{-j}r} \frac{G_+(X,X^+)}{\d(X)} dX & = &
\sum^{N_j}_{j=1} \int_{A^+_j\cap B^+\left(X^j_i, \frac{r}{2^{j-2}}\right)}
\frac{G_+(X,X^+)}{\d(X)} dX \\
& \le & C\sum^{N_j}_{i=1} \frac{\o^{+}(B(Q^j_i, r2^{-j}))}{(r2^{-j})^{n-1}}
\H^n(A^+_j\cap B^+(X^j_i) \nonumber \\
& \le & C\frac{(r2^{-j})^n}{(r2^{-j})^{n-1}} \o^{+}(B(Q,r)) \nonumber \\
& \le & C2^{-j}r\o^{+}(B(Q,r)). \nonumber
\end{eqnarray}
Combining (\ref{eqn-NTA.3}), (\ref{eqn-NTA.5}) and (\ref{eqn-NTA.9}) we
obtain
\begin{eqnarray}\label{eqn-NTA.10}
\int_{B(Q,r)} \left(\frac{G_+(X,X^+)}{\d(X)}\right)^{2+\eta}dX & \le &
C\sum^\infty_{j=0} 2^{j(1+\eta)-j\beta(1+\eta)}2^{-j} \\
& &\qquad \cdot
\left(\frac{G_+(A^+(Q,r),X^{+})}{r}\right)^{1+\eta}\cdot r\o^+(B(Q,r)). \nonumber
\end{eqnarray}
If $\eta<\frac{\beta}{1-\beta}$ the series in the r.h.s. in (\ref{eqn-NTA.10})
converges. The estimate for $G_-(-,X^-)$ is identical. Thus using
(\ref{eqn-NTA.2}), (\ref{eqn1.18}) and (\ref{eqn-NTA.10}) we conclude that
\begin{equation}\label{eqn-NTA.11}
\int_{B(Q,r)} |\nabla u|^{2+\eta} \le
C_\eta\left(\frac{\o^+(B(Q,r))^{2+\eta}}{r^{(n-1)(1+\eta)-1}} +
\frac{\o^-(B(Q,r))^{2+\eta}}{r^{(n-1)(1+\eta)}}.\right)
\end{equation}
Hence $u\in \H^1(B(Q,r))$ and $\O$ is admissible.\qed
\end{proof}

\begin{thm}\label{thm1.3}
Let $\O\subset\RR^n$ be a 2-sided locally NTA domain. Given a compact set $K\subset\RR^n$ there exists
$R_K\in(0,\min\{\delta(X^+),\delta(X^-)\}$ such that for $4r<R_K$ and
$Q\in\po\cap K$
\begin{equation}\label{eqn1.20}
\frac{\o^{\pm}(B(Q,r))}{r^{n-1}} \sim\left(\frac{1}{r^2}
\int_{B(Q,r)}\frac{|\nabla u^{\pm}|^2}{|X-Q|^{n-2}}\, dX\right)^{\frac{1}{2}}
\end{equation}
and
\begin{equation}\label{eqn1.21}
\gamma(Q,r)^{\frac{1}{2}}\sim \frac{\o^+(B(Q,r))}{r^{n-1}} \cdot
\frac{\o^-(B(Q,r))}{r^{n-1}}.
\end{equation}
\end{thm}

The proof is a straightforward combination of the doubling property of
$\o^{\pm}$ (see (\ref{eqn1.19})), (\ref{prel-eqn9}), (\ref{eqn1.17}) and
(\ref{eqn1.18}). The constants that appear (\ref{eqn1.20}) and
(\ref{eqn1.21}) depend on the set $K$.

We turn our attention to the tangent structure of 2-sided locally NTA domains.

Let $\O\subset\RR^n$ be a 2-sided locally NTA domain. Let
$\{r_j\}_{j\ge 1}$ be a sequence of positive numbers such that
$\lim_{j\to\infty} r_j=0$. Consider the domains
\begin{equation}\label{eqn-structure.1}
\O^{\pm}_j= \frac{1}{r_j}(\O^{\pm}-Q)\mbox{ with }\po^{\pm}_j=\frac{1}{r_j}
(\po^{\pm}-Q),
\end{equation}
the functions
\begin{equation}\label{eqn-structure.2}
u^{\pm}_j(X)=\frac{u^{\pm}(r_jX+Q)}{\o^{\pm}(B(Q,r_j))} r^{n-2}_j
\end{equation}
and the measures
\begin{equation}\label{eqn-structure.3}
\o^{\pm}_j(E)=\frac{\o^{\pm}(r_jE+Q)}{\o^{\pm}(B(Q,r_j))}\mbox{ for
}E\subset\RR^n\mbox{ a Borel set.}
\end{equation}
Note that Lemma \ref{lem1.4} ensures that given a compact set
$K\subset\RR^n$ containing $Q$, for $j$ large enough (depending only on
$K$)
\begin{equation}\label{eqn-structure.4}
C^{-1}_K\le \frac{u^{\pm}(A^{\pm}(Q,r_j))}{\o^{\pm}(B(Q,r_j))} r^{n-2}_j
\le C_K.
\end{equation}

Here $C_K$ is a constant that only depends on $K$ and $A^{\pm}(Q,r_j)$
denote the non-tangential points associated to $Q$ at radius $r_j$ in
$\O^{\pm}$.

The boundary Harnack principle (see Lemma \ref{lem1.3}) yields that for
$N>1$, $X\in B(0,N)$ and $j$ large enough depending only on $N$
\begin{equation}\label{eqn-structure.5}
u^{\pm}(r_jX+Q)\le C_{N,K} u^{\pm}(A^{\pm}(Q,r_j)).
\end{equation}
Thus combining (\ref{eqn-structure.2}) and  (\ref{eqn-structure.4}) we obtain that
\begin{equation}\label{eqn-structure.6}
\sup_{j\ge 1} \sup_{X\in B(0,N)} u^{\pm}_j(X)\le C_{N,K}<\infty.
\end{equation}
Furthermore since $\o^{\pm}$ are locally doubling (see Lemma \ref{lem1.5})
\begin{equation}\label{eqn-structure.7}
\sup_{j\ge 1}\o^{\pm}_j(B(0,N))\le C_{N,K}<\infty.
\end{equation}

\begin{thm}\label{thm-structure.1}
Let $\O\subset\RR^n$ be a 2-sided locally NTA domain. 
Using the notation above, we have that there exists a sequence (which we
relabel), satisfying as $j\to\infty$
\begin{eqnarray}
\O^{\pm}_j\to\O^{\pm}_\infty && \mbox{in the Hausdorff distance
sense}\label{eqn-structure.8}\\
			&& \mbox{uniformly on compact sets}\nonumber\\
\po^{\pm}_j\to \po^{\pm}_\infty && \mbox{in the Hausdorff distance
sense}\label{eqn-structure.9}\\
			&& \mbox{uniformly on compact sets}\nonumber
\end{eqnarray}
where $\O^{\pm}_\infty$ are unbounded NTA domains with
$\po^+_\infty=\po^-_\infty$. Moreover, there exist $u^{\pm}_\infty\in
C(\RR^n)$ such that
\begin{equation}\label{eqn-structure.10}
u^{\pm}_j\to u^{\pm}_\infty\mbox{ uniformly on compact sets}
\end{equation}
and
\begin{equation}\label{eqn-structure.11}
\left\{\begin{array}{c@{\hspace{.3in}}l}
\Delta u^\pm_\infty=0&\mbox{in }\O^{\pm}_\infty\\
u^{\pm}_\infty=0&\mbox{on }\po^{\pm}_\infty\\
u^{\pm}_\infty>0&\mbox{in }\O^{\pm}_\infty.
\end{array}\right.
\end{equation}
Furthermore
\begin{equation}\label{eqn-structure.12}
\o^{\pm}_j\rightharpoonup\o^{\pm}_\infty\mbox{ weakly as Radon measures.}
\end{equation}
Here $\o^{\pm}_\infty$ are the harmonic measures of $\O^{\pm}_\infty$ with
pole at infinity, corresponding to $u^{\pm}_\infty$, i.e. $\forall \varphi\in C_c^\infty(\RR^n)$,
\begin{equation}\label{green-infty}
\int_{\O_\infty^\pm}u^\pm_\infty \Delta\varphi=\int_{\po^\pm_\infty}\varphi\, d\o^{\pm}_\infty.
\end{equation} 
\end{thm}
For the proof of this theorem see \cite{KT2} section 4.

When $\O$ is a 2-sided locally NTA domain,  by the differentiation theory of Radon measures
(see \cite{EG}) we know that 
\begin{equation}\label{eqn-decomp-po}
\po=\Lambda_1\cup \Lambda_2\cup \Lambda_3\cup \Lambda_4,
\end{equation}
where
\begin{equation}\label{eqn-l1}
\Lambda_1=\left\{Q\in\po: 0<h(Q):=\frac{d\o^-}{d\o^+}(Q)=D_{\o^+}\o^-(Q)=\lim_{r\to 0}\frac{\o^-(B(Q,r))}{\o^+(B(Q,r))}<\infty\right\}
\end{equation}
\begin{equation}\label{eqn-l2}
\Lambda_2=\left\{Q\in\po: D_{\o^+}\o^-(Q)=\lim_{r\to 0}\frac{\o^-(B(Q,r))}{\o^+(B(Q,r))}=\infty\right\}
\end{equation}
\begin{equation}\label{eqn-l3}
\Lambda_3=\left\{Q\in\po: D_{\o^+}\o^-(Q)=\lim_{r\to 0}\frac{\o^-(B(Q,r))}{\o^+(B(Q,r))}=0\right\}
\end{equation}
\begin{equation}\label{eqn-l4}
\Lambda_4=\left\{Q\in\po: \lim_{r\to 0}\frac{\o^-(B(Q,r))}{\o^+(B(Q,r))}\hbox{ does not exist }\right\}.
\end{equation}
Note that:
\begin{itemize}
\item  $\o^+(\Lambda_2)=0$, $\o^-(\Lambda_3)=0$ and $\o^\pm(\Lambda_4)=0$.
\item $\o^+\perp \o^-$ in $\Lambda_2\cup \Lambda_3$.
\item $\o^+\res \Lambda_1$ and $\o^-\res\Lambda_1$ are mutually absolutely continuous.
\item By the Radon-Nikodym theorem  $h\in L_{loc}^1(\o^+)$ and $\frac{1}{h}\in L_{loc}^1(\o^-)$.
\end{itemize}

Define
\begin{equation}\label{eqn-size.18}
\G=\{Q\in\Lambda_1: \, h(Q)=\lim_{r\to 0}\NotInt_{B(Q,r)}h\, d\o^+,\ \lim_{r\to 0}
\mbox{\NotInt}_{B(Q,r)} |h(P)-h(Q)|d\o^+(P)=0\}.
\end{equation}
Note that $\o^{\pm}(\Lambda_1\backslash\G)=0$.

\begin{thm}\label{thm-structure.2}
Let $\Omega\subset\RR^n$ be a 2-sided locally NTA domain. For $Q\in\G$ (defined in 
(\ref{eqn-size.18})) the blow up
procedure in Theorem \ref{thm-structure.1} yields
\begin{eqnarray}
\o^+_\infty & = & \o^-_\infty \label{eqn-structure.13} \\
u_\infty & = & u^+_\infty-u^-_\infty\mbox{ is a harmonic polynomial in
$\RR^n$}. \label{eqn-structure.14}
\end{eqnarray}
Furthermore there exists $\eta=\eta(n)>0$ such that if $\Omega$ is a $\eta$-Reifenberg
flat domain (i.e for each compact set $K\subset\RR^n$ there exists $r_K>0$ so that for 
$P\in\po\cap K$, and $r\in(0,r_K)$, $\beta_\infty(P,r)<\eta(n)$), then $u_\infty$ is linear.
Here
\begin{equation}\label{prel-eqn3A}
\beta_\infty(P,r)=\frac{1}{r}\inf_{L\in G(n,n-1)} D[\po\cap B(P,r); L\cap B(P,r)],
\end{equation}
and $D$ denotes the Hausdorff distance between sets.
\end{thm}

\begin{proof}
Let $Q\in\G$,  and $\{r_j\}_{j\ge 1}$ a sequence of positive numbers such
that $\lim_{j\to\infty}r_j=0$. Suppose that  (\ref{eqn-structure.8}),
(\ref{eqn-structure.9}), (\ref{eqn-structure.10}), (\ref{eqn-structure.11})
and (\ref{eqn-structure.12}) hold. Let $\varphi\in C_c(\RR^n)$ then
\begin{equation}\label{eqn-structure.15}
\int_{\po^{\pm}_j}\varphi d\o^{\pm}_j = \frac{1}{\o^{\pm}(B(Q,r_j))}
\int_{\po^{\pm}}\varphi\left(\frac{P-Q}{r_j}\right)d\o^+(P).
\end{equation}
In particular if $\spt\varphi\in B(0,M)$ then 
\begin{eqnarray}\label{eqn-structure.16}
\int_{\po^-_j}\varphi d\o^-_j & = & \frac{1}{\o^-(B(Q,r_j))}
\int_{\po}\varphi\left(\frac{P-Q}{r_j}\right)h(P)d\o^+(P) \\
& = & \frac{1}{\NotInt_{B(Q,r_j)}hd\o^+} \cdot \frac{1}{\o^+(B(Q,r_j))}
\int\varphi\left(\frac{P-Q}{r_j}\right)h(P)d\o^+(P)\nonumber \\
& = & \frac{h(Q)}{\NotInt_{B(Q,r_j)}hd\o^+} \cdot \frac{1}{\o^+(B(Q,r_j))}
\int_{\po} \varphi\left(\frac{P-Q}{r_j}\right)d\o^+(P) \nonumber \\
&&+ \frac{1}{\NotInt_{B(Q,r_j)}hd\o^+}\cdot \frac{1}{\o^+(B(Q,r_j))}
\int_{\po} \varphi\left(\frac{P-Q}{r_j}\right)(h(P)-h(Q))d\o^+(P). \nonumber
\end{eqnarray}
Thus using the fact that $\o^+$ is locally doubling
(\ref{eqn-structure.16}) yields
\begin{eqnarray}\label{eqn-structure.17}
 \left|\int_{\po^-_j} \varphi d\o^-_j - \frac{h(Q)}{\NotInt_{B(Q,r_j)}hd\o^+}
\int_{\po^+_j} \varphi d\o^+_j\right| 
 & \le & 
\frac{ \|\varphi\|_{\infty}}{\NotInt_{B(Q,r_j)} hd\o^+}  \nonumber\\
& \cdot&\frac{\o^+(B(Q,Mr_j))}{\o^+(B(Q,r_j))}
\NOTINT_{_{B(Q,Mr_j)}}|h(P)-h(Q)|d\o^+(P)\nonumber \\
& \le & \frac{C_{K,M} \|\varphi\|_{\infty} }{\NotInt_{B(Q,r_j)}hd\o^+}
\NOTINT_{_{B(Q,Mr_j)}}|h(P)-h(Q)|d\o^+(P).\nonumber
\end{eqnarray}
Since $Q\in\G$ letting $j\to\infty$ we obtain
\begin{equation}\label{eqn-structure.18}
\int_{\po^+_\infty}\varphi d\o^-_\infty = \int_{\po^-_\infty} \varphi
d\o^+_\infty
\end{equation}
for every $\varphi\in C_c(\RR^n)$. Since $\po^+_\infty=\po^-_\infty$ to
show that $u_\infty=u^+_\infty-u^-_\infty$ is harmonic in $\RR^n$ let
$\varphi\in C^\infty_c(\RR)$ by (\ref{eqn-structure.18}) we have
\begin{eqnarray}\label{eqn-structure.19}
\int_{\RR^n} u_\infty\Delta\varphi & = & \int_{\O^+_\infty}
u_\infty\Delta\varphi dY - \int_{\O^-_\infty} u_\infty\Delta\varphi d Y \\
& = & \int_{\po^+_\infty} \varphi d\o^+_\infty - \int_{\po^-_\infty}\varphi
d\o^-_\infty=0\nonumber
\end{eqnarray}
Since $u_\infty$ is continuous in $\RR^n$, it is weakly harmonic and
therefore harmonic in $\RR^n$. Note that $u_\infty(0)=0$.
An argument similar to the one that appears in the proof of Theorem 4.4 in
\cite{KT2} shows that $u_\infty$ is a harmonic polynomial. 
Theorem 4.1 in \cite{KT1} shows that 
given $\delta>0$ there exists $\eta>0$ such that if $\O$ is 
$\eta$-Reifenberg flat then 
$\o^+$ is $\delta$-doubling as in Definition 4.4
in \cite{KT2}.  The same argument as in the proof of Theorem 4.4 in \cite{KT2} shows in this 
case that if $\o^+$ is $\delta$ doubling with $n\delta<1$ then $u_\infty$ is linear.\qed
\end{proof}



\begin{cor}\label{cor-structure.1}
There exists $\eta>0$ such that if 
$\O$ is a $\eta$-Reifenberg flat domain then 
\begin{equation}\label{g-size}
\dim_\H\G\le n-1.
\end{equation}
\end{cor}

\begin{proof} Theorem 3.1 in \cite{KT1} shows that if $\eta$ is small enough depending only on $n$ then 
$\O$ is a 2-sided locally NTA domain. Thus by Theorem \ref{thm-structure.2} for $Q\in\G$ all blow-ups
of $\po$ at $Q$ are the zero set of linear polynomial that is an $(n-1)$-plane. 
For $Q\in\po$, the last remark in Theorem \ref{thm-structure.2} ensures that $\lim_{r\to 0}\beta_\infty(Q,r)=0$.
Thus given $\e>0$ there exists $r_{Q,\e}>0$ such that for $r<r_{Q,\e}$;
$\beta_\infty(Q,r)<\e$, which implies that there exists an $(n-1)$ plane $L(Q,r)$
through $Q$ so that
\begin{equation}\label{eqn-size.9A}
\po\cap B(Q,r)\subset\po\cap B(Q,r)\subset(L(Q,r)\cap B(Q,r):\e r).
\end{equation}
Thus for $Q\in \po$, $\lim_{r\to 0}\beta_{\po}(Q,r)=0$. Lemma \ref{lem-size.2}
yields the conclusion of the corollary.\qed
\end{proof}

\section{Tangent structure and size of $\Gamma$}

Let $\cF$ be the set of $(n-1)$ flat measures in $\RR^n$, i.e.
\begin{equation}\label{eqn-TS.A.00}
\cF=\{c\H^{n-1}\res V: c\in (0,\infty); V\in G(n,n-1)\}.
\end{equation}
Note that since $G(n,n-1)$ is compact, $\cF$ has a compact basis, and it is
closed under weak convergence of Radon measure.

\begin{lem}\label{lem-TS.2}
Let $h$ be a harmonic polynomial in $\RR^n$ such that $h(0)=0$ and
$\{h>0\}$ and $\{h<0\}$ are unbounded NTA domains. Let $\o$ be the
corresponding harmonic measure, i.e. $\forall\varphi\in
C^\infty_c(\RR^n)$
\begin{equation}\label{eqn-TS.A.18}
\int_{\{h>0\}}h\Delta\varphi = \int_{\{h<0\}} h^-\Delta\varphi =
\int_{\{h=0\}} \varphi d\o.
\end{equation}
There exists $\epsilon_0>0$ (depending on the NTA
constant of $\{h>0\}$ and on $n$) such that if for some $r_0>0$
\begin{equation}\label{eqn-TS.A.19}
d_r(\o,\cF)<\epsilon_0\ \mbox{ for }r\ge r_0,\ \mbox{ then}\ \o\in\cF.
\end{equation}
\end{lem}

\begin{remark}\label{rem-TS.A.1}
Note that $h$ is the Green's function with pole at infinity for $\{h>0\}$
and $\o$ is its corresponding harmonic measure.
\end{remark}

\begin{proof}
Let $\tau>1$ and $r\ge r_0$ there exists $\Psi\in\cF$ such that $F_{\tau
r}(\Psi)=1$ and
\begin{equation}\label{eqn-TS.A.20}
F_r\left(\frac{\o}{F_{\tau r}(\o)}, \Psi\right) \le
F_{r\tau}\left(\frac{\o}{F_{\tau r}(\o)}, \Psi\right)<\epsilon_0.
\end{equation}
Thus
\begin{equation}\label{eqn-TS.A.21}
F_r(\Psi)-\epsilon_0\le \frac{F_r(\o)}{F_{\tau r}(\o)} \le
F_r(\Psi)+\epsilon_0.
\end{equation}
Since $\Psi=c\H^{n-1}\res V$, $F_{r\tau}(\Psi)=1=c\frac{\o_n}{n+1} (\tau
r)^n$ and $F_r(\Psi)=\tau^{-n}$. Thus given $\delta>0$ (small enough) for
$\tau\in (1,\tau_{\epsilon_0,\delta})$ with $\tau_{\epsilon_0,\delta}=
\left(\frac{\delta \epsilon^{-1}_0}{2}\right)^{\frac{1}{n}}$ for $r\ge r_0$
(\ref{eqn-TS.A.21}) yields
\begin{equation}\label{eqn-TS.A.22}
(1+\delta)^{-1}\tau^{-n}<\frac{F_r(\o)}{F_{\tau_r(\o)}}<(1+\delta)\tau^{-n}.
\end{equation}

Applying (\ref{eqn-TS.A.22}) to $\tau^j r$ for $j=1,\cdots, \ell$ with $\ell\in\NN$ and $r\ge r_0$,
then multiplying the outcomes we obtain
\begin{equation}\label{eqn-TS.A.23}
[(1+\delta)^{-1}\tau^{-n}]^\ell\le \frac{F_r(\o)}{F_{\tau^\ell r}(\o)}\le
[(1+\delta)\tau^{-n}]^\ell.
\end{equation}
Since $\o$ is a doubling measure with doubling constant depending only on
the NTA constant of $\{h>0\}$ and on $n$ (see \cite{KT3} Lemma 3.1 or
\cite{JK} Lemma 4.9, 4.11) from the definition of $F_r$ (see Definition
\ref{defn-GMTP.1}) we have that there is $C>1$ such that for $r>0$
\begin{equation}\label{eqn-TS.A.24}
C^{-1}r\o (B(0,r))\le
\frac{r}{2}\o\left(B\left(0,\frac{r}{2}\right)\right)\le F_r(\o)\le r\o
(B(0,r)).
\end{equation}

Combining (\ref{eqn-TS.A.23}) and (\ref{eqn-TS.A.24}) we obtain
\begin{equation}\label{eqn-TS.A.25}
C^{-1}[(1+\delta)^{-1}\tau^{-n}]^\ell \le
\frac{\o(B(0,r))}{\tau^\ell\o(B(0,\tau^\ell r))} \le C[(1+\delta)\tau^{-n}]^\ell.
\end{equation}
Thus
\begin{equation}\label{eqn-TS.A.26}
C^{-1}(1+\delta)^{-\ell}\frac{\o(B(0,\tau^\ell r))}{(\tau^\ell r)^{n-1}} \le
\frac{\o(B(0,r))}{r^{n-1}} \le C(1+\delta)^\ell
\frac{\o(B(0,\tau^\ell r))}{(\tau^\ell r)^{n-1}}.
\end{equation}

By Lemma 3.4 in \cite{KT3} (see also Lemma 4.8 in \cite{JK}) we know
that there exists $C>1$ depending  only on $n$ and on the NTA constant of
$\{h>0\}$ such that
\begin{equation}\label{eqn-TS.A.27}
C^{-1}\le \frac{\o(B(0,r))}{r^{n-2}h(A(0,r))} \le C.
\end{equation}
Here $A(0,r)\in \{h>0\}$ denotes a nontangential point for $0$ at radius $r>0$.
Combining (\ref{eqn-TS.A.26}) and (\ref{eqn-TS.A.27}) we have
\begin{equation}\label{eqn-TS.A.28}
C^{-1}(1+\delta)^{-\ell} \frac{h(A(0,\tau^\ell r))}{\tau^\ell r} \le \frac{h(A(0,r))}
{r} \le C(1+\delta)^\ell \frac{h(A(0,\tau^\ell r))}{\tau^\ell r}.
\end{equation}
If $1+\delta=\tau^\beta$ with $\beta\in (0,1)$ then (\ref{eqn-TS.A.28})
becomes
\begin{equation}\label{eqn-TS.A.29}
C^{-1}\tau^{-\beta \ell} \frac{h(A(0,\tau^\ell r))}{\tau^\ell r} \le
\frac{h(A(0,r))}{r} \le C\tau^{\beta\ell} \frac{h(A(0,\tau^\ell r))}{\tau^\ell r}.
\end{equation}

Note that by choosing $\delta=4\epsilon_0$ (with $\epsilon_0>0$ to still be
determined) then $\tau_{\epsilon_0,\delta} = \tau_0=2^{\frac{1}{n}}$ and 
$1+\delta=1+4\epsilon_0=\tau^{\beta/n}$ for some $\tau\in(1, 2^{1/n})$ and
$\beta\in (0,1)$ provided $\epsilon_0<\frac{1}{4} (2^{1/n}-1)$. 
For $s\in (0,\tau r_0)$ there is $\ell\ge 1$ such that 
$\tau^{\ell-1}r_0<s\le\tau^\ell r_0$. For such $s$, the
boundary Harnack's inequality (for NTA domains (see Lemma 3.3 \cite{KT3},
also Lemma 4.4 \cite{JK})), combined with (\ref{eqn-TS.A.29}) yields
\begin{eqnarray}\label{eqn-TS.A.30}
\frac{h(A(0,s))}{s} & \le & C\frac{h(A(0,\tau^\ell r_0))}{\tau^{l-1}r_0} \le
C\tau \frac{h(A(0,\tau^\ell r_0))}{\tau^\ell r_0} \\
& \le & C\tau \tau^{\ell\beta} \frac{h(A(0,r_0))}{r_0} \nonumber \\
& \le & C\tau^{2+\beta} \left(\frac{s}{r_0}\right)^{\beta }
\frac{h(A(0,r_0))}{r_0}.\nonumber
\end{eqnarray}
Since $h$ is harmonic using its Poisson integral formula and computing its
second derivatives (as in the proof of Theorem 4.4 in \cite{KT2}) from
(\ref{eqn-TS.A.30}) we obtain that for $X\in B(0,s)$
\begin{equation}\label{eqn-TS.A.31}
|\partial_{\alpha_1}\partial_{\alpha_2} h(X)| \le \frac{h(0,s)}{s^2} \le
C(\tau,r_0)s^{\beta-1} \frac{h(A(0,r_0))}{r_0}.
\end{equation}
Since $\beta<1$ letting $s\to\infty$ we conclude that $h$ is a polynomial of
degree 1, and therefore $\o$ is an $(n-1)$ flat measure.
\qed
\end{proof}

We will now return to the question of the extent to which the relative
behavior of the interior and exterior harmonic measures determines the size
of the boundary of a domain.

\begin{remark}\label{rem-oo}
Note that for $Q\in\G$
\begin{equation}\label{eqn-oo}
\Tan(\o^+,Q)=\Tan(\o^-,Q)
\end{equation}
\end{remark}

\begin{thm}\label{lem-TS.A.3}
Let $\Omega$ be a 2-sided locally NTA domain. Let $\Gamma$ be as in (\ref{eqn-size.18}), and
\begin{equation}\label{eqn-TS-g-ast}
\Gamma^\ast=\left\{Q\in \Gamma:\ \Tan(\o^\pm,Q)\cap \cF\not = \emptyset\right\}.
\end{equation} 
Then for $Q\in\Gamma^\ast$, $\Tan(\o^\pm, Q)\subset\cF$.
In particular, all blow-ups of $\po$ at $Q\in\Gamma^\ast$
are $(n-1)$-planes, and $\dim_\H \, \Gamma^\ast\le n-1$. Furthermore
$\Gamma_0=\Gamma\backslash\Gamma^\ast$ satisfies
$\o^\pm(\Gamma_0)=0$.
\end{thm}

\begin{proof}
For $Q\in\Gamma$ the blow-up procedure described in Theorem \ref{thm-structure.1}
always yields a harmonic polynomial (see Theorem \ref{thm-structure.2}). Let $h$ be a tangent harmonic polynomial of $u$ at $Q$, with
$\{h>0\}$ $\{h<0\}$ (unbounded NTA domains) and $\nu$
the corresponding harmonic measures to $h^\pm$. By \cite{HS} the zero set
of $h$, i.e. $\partial\{h>0\}$ decomposes into a
disjoint union of the embedded $C^1$ submanifold
$h^{-1}\{0\}\cap\{|Dh|>0\}$, together with a closed set $h^{-1}\{0\}\cap
|Dh|^{-1}|0|$ which is countably $(n-2)$-rectifiable. Furthermore by Lemma \ref{lem-GMTP.3}, 
$\spt\nu=h^{-1}\{0\}$. For $Y\in h^{-1}\{0\}\cap \{|Dh|>0\}$ and $X\in
\RR^n$
\begin{equation}\label{eqn-TS.A.32}
h_{Y,r}(X)=\frac{h(rX+Y)}{r}
\mathop{\longrightarrow}\limits_{r\to 0}\lan Dh(Y),X\ran
\end{equation}
uniformly on compact sets. Thus $r^{-1}(\partial\{h>0\}-Y)\to
\lan\frac{Dh(Y)}{|Dh(Y)|}\ran^{\perp} =V$ as $r\to 0$, in the
Hausdorff distance sense and $r^{-(n-1)}T_{Y,r}[\nu]\to
|Dh(Y)|\H^{n-1}\res V$. Therefore, for $Y\in
h^{-1}\{0\}\cap\{|Dh|>0\}$ all non-zero tangent measures of $\nu$ at
$Y$ are flat, i.e. $\Tan(\nu,Y)\subset\cF$. By Theorem \ref{thm-GMTP.3}
for $\o=\o^\pm$ a.e. $Q\in\Gamma$ if $\nu\in\Tan(\o,Q)$, then for all
$Y\in\spt\nu$, $\Tan(\nu,Y)\subset\Tan(\o,Q)$.
Thus, for $\o$ a.e. $Q\in\Gamma$, $\cF\cap\Tan(\o,Q)\ne\emptyset$, which proves that $\o^\pm(\Gamma_0)=0$. Our goal
is to use Corollary \ref{cor-GMTP.1} combined with Lemma \ref{lem-TS.2} to
show that for $Q\in\Gamma^\ast$, $\Tan(\o,Q)\subset\cF$. Let
$\cM=\cF\cup\Tan(\o,Q)$. Recall that $\cF$ the set of all $(n-1)$ flat
measures is a $d$-cone with compact basis. Since $\o$ is a doubling Radon
measure Theorem \ref{thm-GMTP.2} ensures that for $Q\in\Gamma$, $\Tan(\o,Q)$
is a $d$-cone with compact basis. Hence $\cM$ is also a $d$-cone with
compact basis. Moreover $\cF\subset\cM$, and $\cF$ is relatively closed
with respect to weak convergence of Radon measures. By Lemma
\ref{lem-TS.2} there exists $\epsilon_0>0$ such that if
$d_r(\mu,\cF)<\epsilon_0$ for all $r\ge r_0$, then $\mu\in\cF$. Corollary
\ref{cor-GMTP.1} ensures then that for $Q\in\Gamma^\ast$,
$\Tan(\o,Q)\subset\cF$. Lemma \ref{lem-GMTP.3} guarantees that all blow ups
of $\po$ at $Q$ converge in the Hausdorff distance sense to an $(n-1)$-plane. 
Thus for $Q\in\Gamma^\ast$, $\lim_{r\to\infty} \beta_\infty(Q,r)=0$. As in the
proof of Corollary (\ref{cor-structure.1}) this implies that for
$Q\in\Gamma^\ast$ $\lim_{r\to 0} \beta_\Gamma^\ast(Q,r)=0$. By Lemma \ref{lem-size.2} we
conclude that  $\dim_\H \Gamma^\ast\le n-1$.
\qed
\end{proof}

\begin{cor}\label{cor-bdy-decomp}
Let $\Omega$ be a 2-sided locally NTA domain. Then the boundary of $\O$ can be decomposed as
follows:
\begin{equation}\label{eqn-decomp1}
\po=\G^\ast\cup S\cup N,
\end{equation}
\begin{equation}\label{eqn-decomp2}
\o^+\res\G^\ast \ll \o^-\res\G^\ast\ll \o^+\res\G^\ast, \ \ 
\o^+\perp \o^-\ \hbox{ in }\ S,\ \hbox {and }\ \o^+(N)=\o^-(N)=0.
\end{equation}
Moreover
\begin{equation}\label{eqn-decomp3}
\dim_{\H}\G^\ast\le n-1.
\end{equation}
Furthermore, if $\o^\pm(\G^\ast)>0$ then 
\begin{equation}\label{eqn-decomp4}
\dim_{\H}\G^\ast= n-1.
\end{equation}
Here $\G^\ast$ is as in Theorem \ref{lem-TS.A.3}, $S=\Lambda_2\cup\Lambda_3$ (see (\ref{eqn-l2}) and (\ref{eqn-l3})), and $N=\Lambda_1\backslash \G^\ast\cup \Lambda_4$.
\end{cor}

\begin{proof}
We only need to show that that (\ref{eqn-decomp4}) holds whenever $\o^\pm(\G^\ast)>0$.
By (\ref{prel-eqn13B}) for $Q_0\in\po$, $r_0<1/8\min\{\delta(X^+), \delta(X^-)\}$, $Q\in \Gamma^\ast\cap B(Q_0, \frac{r_0}{2})$ and $0<r<r_0$
\begin{eqnarray}\label{eqn-TS.35}
\frac{\o^+(B(Q,r))}{r^{n-1}}\cdot \frac{\o^-(B(Q,r))}{r^{n-1}} &\le & C(Q_0,r_0)\\
\left(\NOTINT_{_{B(Q,r)}}h\, d\o^+\right)\left(\frac{\o^+(B(Q,r))}{r^{n-1}}\right)^2&\le & C(Q_0, r_0)\nonumber
\end{eqnarray}
Thus
\begin{equation}\label{eqn-TS.36}
\frac{\log\left(\NOTINT_{_{B(Q,r)}}h\, d\o^+\right)^{1/2}}{\log r} +\frac{\log\o^+(B(Q,r))}{\log r}\ge
n-1 +\frac{\log C(Q_0,r_0)}{\log r}.
\end{equation}
Letting $r$ tend ot 0 in (\ref{eqn-TS.36}) we obtain that
\begin{equation}\label{eqn-TS.37}
\liminf_{r\to 0}\frac{\log\o^+(B(Q,r))}{\log r}\ge n-1.
\end{equation}
By Proposition 2.3 in \cite{F} from (\ref{eqn-TS.37}) we conclude that, since $\o^+(\Gamma^\ast)>0$,
 $\dim_\H\Gamma^\ast \ge n-1$.\qed
\end{proof}

\begin{thm}\label{rectif}
Let $\O$ be a 2-sided locally NTA domain such that $\H^{n-1}\res\po$ is a Radon measure. Then 
as in Theorem \ref{lem-TS.A.3}, $\po=\G^\ast\cup S\cup N$ and $\G^\ast$ is 
$(n-1)$-rectifiable.
\end{thm}

\begin{proof}
Our strategy consists in proving that the density of $\H^{n-1}\res \G^\ast$ exists and is 1 
a.e.. Then we appeal to Theorem 17.6 in \cite{M}, which provides a rectifiability criteria.

First we prove that for $Q\in\Gamma^\ast$ (see (\ref{eqn-TS-g-ast}) for the definition)
\begin{equation}\label{eqn-RP4}
\Theta_\ast^{n-1}(\H^{n-1}\res\po, Q)=\liminf_{r\to 0}\frac{\H^{n-1}(B(Q,r)\cap\po)}{\o_{n-1}r^{n-1}}\ge 1.
\end{equation}

For $Q\in\G^\ast$ and $\d>0$ by Theorem \ref{lem-TS.A.3} there exists $r_0>0$ so that for $r<r_0$
there exists $L(Q,r)$ an $(n-1)$ plane containing $Q$ so that 
\begin{equation}\label{eqn-size.3}
\frac{1}{r}D[\po\cap B(Q,r); L(Q,r)\cap B(Q,r)]\le\d.
\end{equation}
Since $\O^\pm$ satisfy the corckscrew condition, we may assume that for 
$r<r_0$ there exist
$A^{\pm}(Q,r)\subset\O^{\pm}$ so that 
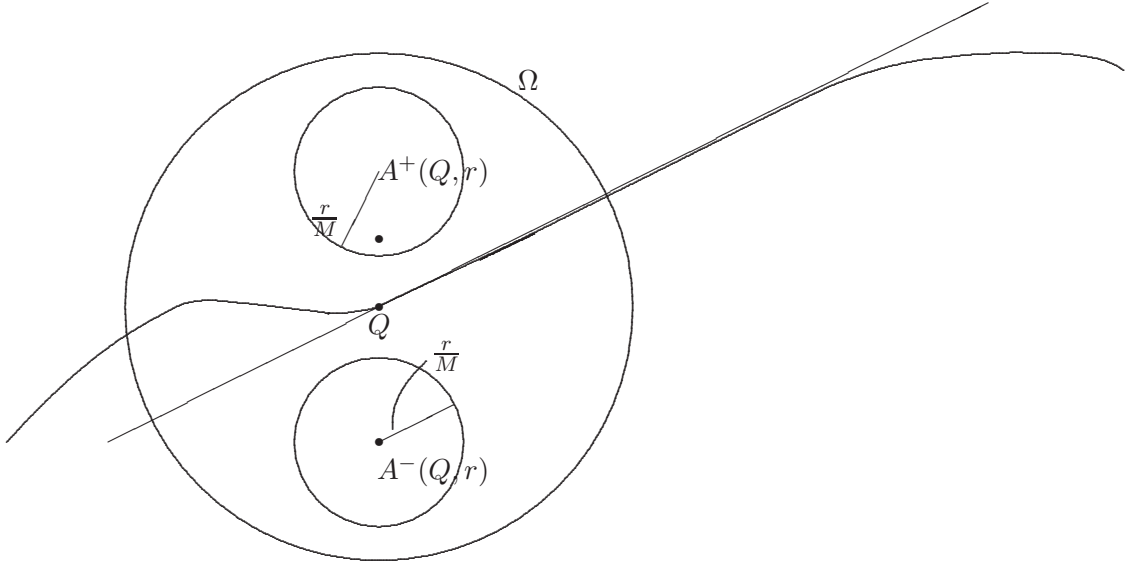
\begin{equation}\label{eqn-size.4}
B\left(A^{\pm}(Q,r),\frac{r}{M}\right)\subset\O^{\pm}\cap B(Q,r).
\end{equation}
\begin{center}
\setlength{\unitlength}{.9mm}
\begin{picture}(155,85)(0,-20)
\put(15,0){\line(2,1){130}}
\put(55,20){\bigcircle{75}}
\put(55,40){\bigcircle{25}}
\put(55,40){\line(-1,-2){5.5}}
\put(55,0){\bigcircle{25}}
\put(55,0){\line(2,1){11}}
\curve(62,12, 58,7, 57,2)
\curve(0,0, 13,12.5, 25,20, 30,21, 40,20, 47.5,19 , 55,20, 60,22.5, 70,27,
80,32, 110,47, 122.5,53, 135,56.5, 155,57.5, 165,55)
\put(55,20){\circle*{1}}
\put(55,30){\circle*{1}}
\put(55,0){\circle*{1}}
\put(77,55){\makebox(0,0)[t]{$\O$}}
\put(47,35){\makebox(0,0)[t]{$\frac{r}{M}$}}
\put(63,40){\makebox(0,0)[m]{$A^+(Q,r)$}}
\put(55,17){\makebox(0,0)[m]{$Q$}}
\put(65,15){\makebox(0,0)[t]{$\frac{r}{M}$}}
\put(63,-2){\makebox(0,0)[t]{$A^-(Q,r)$}}
\end{picture}
\end{center}

If $\vec{n}(Q,r)$ denotes the unit normal to $L(Q,r)$ (\ref{eqn-size.3})
and (\ref{eqn-size.4}) ensure that for $\d$ small $(\d<\frac{1}{2M})$
\[
|\langle A^{\pm}(Q,r)-Q, \vec{n}(Q,r)\ran|\ge 2\d r.
\]
We may assume that $\lan A^+(Q,r)-Q, \vec{n}(Q,r)\ran\ge 2\d r$. If $Z\in
B(Q,r)$ and $\lan Z-Q, \vec{n}(Q,r)\ran\ge 2\d r$ then $Z\in\O^+$,
otherwise $Z\in\O^-$ (since  $Z\not\in\po$ by (\ref{eqn-size.3})) and by
connectivity there would be a point $P\in\po$ in the segment joining
$A^+(Q,r)$ to $Z$. Such $P$ would satisfy $\lan P-Q, \vec{n}(Q,r)\ran\ge
2\d r$ which contradicts (\ref{eqn-size.3}). This proves that
\begin{equation}\label{eqn-size.5}
\{Z\in B(Q,r):\lan Z-Q, \vec{n}(Q,r)\ran\ge 2\d r\}\subset\O^+\cap B(Q,r)
\end{equation}
and
\begin{equation}\label{eqn-size.6}
\{Z\in B(Q,r): \lan Z-Q; \vec{n}(Q,r)\ran\le -2\d r\}\subset\O^-\cap
B(Q,r).
\end{equation}
Thus for $x\in L(Q,r)\cap B\left(Q,r\sqrt{1-4\d^2}\right)$
a simple connectivity argument shows that there exists $P\in\po$ such that
$P=(x,t)$ with $|t|<\d r$. Hence $P\in\po\cap B(Q,r)$. If $\pi_{Q,r}$
denotes the orthogonal projection onto $L(Q,r)$ we have for $\d$ small
enough
\begin{eqnarray}\label{eqn-size.7}
\H^{n-1}(\po\cap B(Q,r)) 
& \ge & \H^{n-1}(\Pi_{Q,r}(\po\cap B(Q,r)) \\
& \ge & \o_{n-1}r^{n-1}(1-4\d^2)^{\frac{n-1}{2}}\nonumber  \\
& \ge & \o_{n-1}r^{n-1}(1-\d),\nonumber
\end{eqnarray}
which ensures that (\ref{eqn-RP4}) holds.
Since $\H^{n-1}\res\po$ is a Radon measure
for $\H^{n-1}\res\po$ -a.e. $Q$
\begin{equation}\label{eqn-RP5}
\Theta^{\ast, n-1}(\H^{n-1}\res\po, Q)=\limsup_{r\to 0}\frac{\H^{n-1}(B(Q,r)\cap\po)}{\o_{n-1}r^{n-1}}\le 1,
\end{equation}
see \cite{M} Theorem 6.2. Thus combining (\ref{eqn-RP4}) and (\ref{eqn-RP5}) we conclude that
for $Q\in\G^\ast$,
\begin{equation}\label{eqn-RP6}
\Theta^{n-1} (\H^{n-1}\res\po, Q)=\lim_{r\to 0}\frac{\H^{n-1}(B(Q,r)\cap\po)}{\o_{n-1}r^{n-1}}= 1.
\end{equation}
Thus since $\H^{n-1}\res\po$ is a Radon measure by Corollary 6.3 in \cite{M} for $\H^{n-1}$-a.e.
$Q\in \G^\ast$
\begin{equation}\label{eqn-RP7}
\Theta^{n-1} (\H^{n-1}\res\po, Q)= \Theta^{n-1} (\H^{n-1}\res\G^\ast, Q)=1.
\end{equation}
Therefore Theorem 17.6 in \cite{M} 
ensures that $\G^\ast$ is $(n-1)$-rectifiable.\qed
\end{proof}

The following theorem proves that there are no Wolff snowflakes for which
$\o^+$ and $\o^-$ are mutually absolutely continuous, answering a question in \cite{LVV}.

\begin{thm}\label{thm-TS.A.1}
Let $\O$ be a 2-sided locally NTA domain. Assume that $\o^+$ and $\o^-$ are
mutually absolutely continuous, then
\begin{equation}\label{eqn-TS.A.34}
\H-\dim\o^+= \H-\dim\o^-= n-1.
\end{equation}
Here the Hausdorff dimension of $\o^\pm$, $\H-\dim\o^\pm$ is defined as in (\ref{eqn-TS.A.33}).
\end{thm}

\begin{proof}
Since $\o^+$ and $\o^-$ are
mutually absolutely continuous it is easy to see that $\H-\dim\o^+= \H-\dim\o^-$.
For each compact set $K\subset\RR^n$, $\o^\pm(\Gamma\cap K)=\o^{\pm}(\po\cap K)$,
hence by for $\Gamma^\ast=\Gamma\backslash\Gamma_0$ with $\Gamma_0$ as in
Theorem \ref{lem-TS.A.3}, $\o^\pm(\Gamma^\ast\cap K)=\o^{\pm}(\po\cap K)$,
and $\dim_\H\Gamma^\ast \le n-1$, i.e. $\forall k>n-1$,
$\H^k(\Gamma^\ast)=0$ which implies that $\H-\dim\o^+\le n-1$ and
$\H-\dim\o^-\le n-1$. Since in this case $\o^\pm(\G^\ast)>0$ (by (\ref{eqn-decomp1}) and (\ref{eqn-decomp2})), (\ref{eqn-decomp4}) yields (\ref{eqn-TS.A.34}).
\qed
\end{proof} 

We conclude by having a second look at $\G$ motivated by the 2-dimensional results in 
Chapter VI of \cite{GM}. Denote by $\o=\o^\pm$, and define
\begin{equation}\label{eqn-RP-g}
\Gamma_g=\left\{Q\in\Gamma:\, 0<\limsup_{r\to 0}\frac{\o(B(Q,r))}{r^{n-1}}<\infty\right\}, \qquad
\G^\ast_g=\G_g\cap\G^\ast,
\end{equation}
\begin{equation}\label{eqn-RP-b}
\Gamma_b=\left\{Q\in\Gamma:\, \limsup_{r\to 0}\frac{\o(B(Q,r))}{r^{n-1}}=0\right\}\qquad
\G^\ast_b=\G_b\cap\G^\ast.
\end{equation}
Since $\o^+\res\Lambda_1$ and $\o^-\res\Lambda_1$ are mutually absolutely continuous, and $\Gamma\subset \Lambda_1$, see (\ref{eqn-size.18}), $\Gamma_g$ and $\Gamma_b$ are well defined. 

By (\ref{prel-eqn13B}) for $Q_0\in\po$, $r_0<1/8\min\{\delta(X^+), \delta(X^-)\}$, $Q\in \Gamma\cap B(Q_0, \frac{r_0}{2})$ and $0<r<r_0$
\begin{equation}\label{eqn-RP1}
\left(\NOTINT_{_{B(Q,r)}}h\, d\o^+\right)\left(\frac{\o^+(B(Q,r))}{r^{n-1}}\right)^2\le C(Q_0, r_0).
\end{equation}
Thus for $Q\in\Gamma$, $\limsup_{r\to 0}\frac{\o(B(Q,r))}{r^{n-1}}<\infty$ which ensures that 
$\Gamma=\Gamma_g\cup \Gamma_b$.

\begin{lem}\label{lem-s-finite}
Let $\O$ be a 2-sided locally NTA domain. Then $\H^{n-1}\res\G_g$ and $\o\res\G_g$ are mutually absolutely
continuous.  In particular $\H^{n-1}\res\G_g$ is $\sigma$-finite.  Furthermore $\G^\ast=\G^\ast_g\cup
\G^\ast_b\cup Z$ with $\o(Z)=0$. 
Moreover if for $E\subset\RR^n$
Borel, $\o(\G^\ast_b\cap E)>0$, then 
$\H^{n-1}(\G^\ast_b\cap E)=\infty$. 
\end{lem}

\begin{proof}
Let
\begin{equation}\label{eqn-RP2}
\G_g=\bigcup_{i= 1}^\infty\G_g^i=\bigcup_{i= 1}^\infty\left\{Q\in \G_g;\,  2^{-i}\le \limsup_{r\to 0}\frac{\o(B(Q,r))}{r^{n-1}}\le 2^i\right\}.
\end{equation}
By Proposition 2.2 in \cite{F} for any Borel set $E\subset \G$ and $i,\, k\in\NN$ 
\begin{equation}\label{eqn-RP3}
2^{-i}\o(E\cap\G^i_g\cap B(0,k))\le \H^{n-1}(E\cap\G^i_g\cap B(0,k))\le 2^{n+i}\o(E\cap\G^i_g\cap B(0,k)),
\end{equation}
which proves the statements that $\H^{n-1}\res\G_g$ and $\o\res\G_g$ are mutually absolutely
continuous. The statement about $\G_b$ is a simple consequence of Proposition 2.2 in \cite{F}.\qed
\end{proof}

\begin{cor}\label{struct-sum}
Let $\O$ be a 2-sided locally NTA domain. Then the boundary of $\O$ can be decomposed as
follows:
\begin{equation}\label{eqn-decomp1A}
\po=\G^\ast_g\cup\G^\ast_b\cup S\cup \widetilde N,
\end{equation}
where
\begin{equation}\label{eqn-decomp2A}
\o^+\ll \o^-\ll \o^+ \ \hbox{ in }\ \G^\ast_g\cup\G^\ast_b,\ \ 
\o^+\perp \o^-\ \hbox{ in }\ S,\ \hbox {and }\ \o^+(\widetilde N)=\o^-(\widetilde N)=0.
\end{equation}
On $\G^\ast_g$, $\H^{n-1}$ is $\sigma$-finite, and $\o$ and $\H^{n-1}$ are mutually
absolutely continuous. On $\G^\ast_b$ for any Borel set $E$, if $\o(\G^\ast_b\cap E)>0$, then 
$\H^{n-1}(\G^\ast_b\cap E)=\infty$. 
\end{cor}

\end{document}